\documentclass[10pt]{article}
\usepackage{amssymb,amsmath,textcomp}
\usepackage{verbatim}
\usepackage{enumerate}
\usepackage{hyperref}
\usepackage{mathrsfs}
\usepackage{amsthm}
\usepackage{xcolor}
\newtheorem{satz}{Theorem}[section]
\newtheorem{lemma}{Lemma}[section]

\newtheorem{bemerk1}{Remark}[section]

\newtheorem{prop}{Proposition}[section]

\newcommand{\iR}{\mathbb{R}}
\newcommand{\iN}{\mathbb{N}}

\newcommand{\oH}{\hspace*{0.39em}\raisebox{0.6ex}{\textdegree}\hspace{-0.72em}H}

\newcommand{\dd}{d\mu(\alpha)}

\newcommand{\ki}{k_{1}}
\newcommand{\kjr}{k_{1}(\Phi(r))}

\newcommand{\supp}{\mathop{\mathrm{supp}}\limits}

\newcommand*{\norm}[1]{\left\Vert{#1}\right\Vert}
\newcommand*{\abs}[1]{\left\vert{#1}\right\vert}

\newcommand*{\Om}{\Omega}

\newcommand*{\izi}{\int_{0}^{\infty}}
\newcommand*{\izt}{\int_{0}^{t}}
\newcommand*{\izj}{\int_{0}^{1}}

\newcommand*{\al}{\alpha}

\newcommand*{\vf}{\varphi}

\newcommand*{\ve}{\varepsilon}

\newcommand{\eqq}[2]{\begin{equation}  #1  \label{#2}\end{equation}    }
\newcommand{\hd}{\hspace{0.2cm}}

\newcommand{\m}[1]{\mbox{#1}}

\newcommand{\jd}{\frac{1}{2}}

\newcommand*{\esssup}{\mathop{\mathrm{ess\hspace{0.05cm}sup}}\limits}

\newcommand{\vdr}{\Phi(r)}

\newcommand{\eqnsl}[2]{
\begin{equation}
\label{#2}
\begin{split}
#1
\end{split}
\end{equation}
}

\newcommand{\vsi}{\varsigma}

\newcommand{\nic}[1]{ }

\DeclareMathOperator*{\esup}{ess\,sup}
\DeclareMathOperator*{\einf}{ess\,inf}

\def\XXint#1#2#3{{\setbox0=\hbox{$#1{#2#3}{\int}$ }
\vcenter{\hbox{$#2#3$ }}\kern-.6\wd0}}

\begin{document}
\begin{center}
{\bf\Large On the Harnack inequality for time-fractional and more general non-local in time subdiffusion equations}
\renewcommand{\thefootnote}{\fnsymbol{footnote}}
\footnote[1]{The first author was partly supported by the grant Sonata Bis UMO-2020/38/E/ST1/00469, National Science Centre, Poland. }
\end{center}
\vspace{0.7em}
\begin{center}
Katarzyna Ryszewska*, Rico Zacher${}^{\diamond}$
\end{center}

\vspace{1cm}
{\footnotesize \noindent \nic{{\bf address:} }*Department of Mathematics and Information Sciences\\
Warsaw University of Technology\\
Koszykowa 75, 00-662 Warsaw, Poland \\
Interdisciplinary Centre for Mathematical and Computational Modelling\\
University of Warsaw\\
Tyniecka 15/17, 02-630 Warsaw, Poland \\
Katarzyna.Ryszewska@pw.edu.pl, \\
Corresponding author\\
\\
${}^{\diamond}$Institute of Applied Analysis \\
University of Ulm \\
89069 Ulm, Germany \\
rico.zacher@uni-ulm.de \\
} \vspace{0.7em}
\begin{abstract}
In this paper we establish the Harnack inequality for globally positive local solutions to a general class of nonlocal in time subdiffusion equations in one space dimension, which includes time-fractional diffusion equations with time order less than one. It is already known that for these equations the classical Harnack inequality does not hold if the space dimension is greater than or equal to two. Here, we complete the analysis, by providing a positive result in one space dimension.
\end{abstract}
\vspace{0.7em}
\begin{center}
\end{center}

\section{Introduction and main results}
Harnack inequalities play a significant role in the analysis of elliptic and parabolic partial differential equations. Since the Harnack inequality (in its original form) is a local estimate, establishing it for solutions to non-local problems appears to be challenging. In recent years, this subject has been attracting considerable attention. As to Harnack type results for parabolic problems with fractional Laplacian or even more general spatially non-local operators, there is already an extensive literature (see \cite{BL,CaSi,CK, Kass1, Kass2, Kass3, Str} and the references therein). Recently, in  \cite{Kass1}, Kassmann and Weidner established the parabolic Harnack inequality for globally nonnegative functions which are local weak solutions to nonlocal equations of the form
\eqq{
\partial_t u(t,x) +p.v.\int_{\mathbb{R}^{N}}(u(t,x)-u(t,y))K(t;x,y)dy = 0,
}{space}
where $K$ is a symmetric kernel satisfying
\[
\lambda (2-\al)|x-y|^{-N-\al} \leq K(t;x,y) \leq \Lambda (2-\al)|x-y|^{-N-\al}
\]
for some $\lambda,\Lambda > 0$, $\al \in [\al_0,2)$, $\al_0 \in (0,2)$.
Their purely analytic proof relies on a careful local boundedness estimate for
nonnegative subsolutions and an improved version of
the weak Harnack inequality for nonnegative supersolutions.
Here it is crucial that the unavoidable (bad) tail term in the
subsolution estimate is in such a form that it can be
controlled by the good tail term appearing in the weak Harnack
estimate.

In the present paper, we adapt many of the ideas from
\cite{Kass1} to answer the open question about the Harnack inequality for globally nonnegative local solutions to
non-local in time subdiffusion equations of the form
\begin{equation} \label{MProb}
\partial_t \big(k*(u-u_0)\big)-\mbox{div}\,\big(A(t,x)Du\big) =0,\quad t\in (0,T),\,x\in
\Omega.
\end{equation}
Here $T>0$, $\Omega$ denotes a bounded domain in $\mathbb{R}^{N}$, $N \in \iN$, $u_0$ plays the role of the initial datum and $A(t,x)=(a_{ij}(t,x))_{i,j=1}^N$ is a given bounded and uniformly elliptic matrix with merely measurable coefficients. By $f_1\ast f_2$ we understand the
convolution on the positive half-line with respect to time, that is
$(f_1\ast f_2)(t)=\int_0^t f_1(t-\tau)f_2(\tau)\,d\tau$, $t\ge 0$, and $Du$  denotes the gradient with respect to the spatial variables.
Concerning the kernel $k$, we assume, among others, that it is of type $\mathscr{PC}$, that is $k\in L_{1,\,loc}(\iR_+)$ is nonnegative and nonincreasing, and there
exists a nonnegative kernel $l\in L_{1,\,loc}(\iR_+)$ such that $k\ast l=1$ in
$(0,\infty)$. The most prominent example of kernel admissible in our setting is given by the
Riemann-Liouville kernel
\begin{equation} \label{mainexample}
    k(t)=\frac{t^{-\alpha}}{\Gamma(1-\alpha)},\quad t>0,\;
    \;\mbox{with some}\;\alpha\in (0,1),
\end{equation}
for which equation \eqref{MProb}, specialised further by taking $A=I\in \iR^{N\times N}$, becomes the time-fractional
diffusion equation
\begin{equation} \label{timefracdiff}
    \partial_t^\alpha(u-u_0)-\Delta u=0
    ,\quad t\in (0,T),\,x\in
\Omega,
\end{equation}
where $\partial_t^\alpha$ is the Riemann-Liouville
fractional derivative.

Although precise weak Harnack estimates for nonnegative weak supersolutions of \eqref{MProb} in all space dimensions have been recently obtained in \cite{krz} for a broad class of $\mathscr{PC}$-kernels, the question about the validity of the (full) Harnack inequality remained partially open. To the best of our knowledge, the problem of the Harnack inequality for equation (\ref{MProb}) has only been studied in the special case of the time-fractional diffusion equation \eqref{timefracdiff}. By exploiting properties of the
fundamental solution to \eqref{timefracdiff}, it was proven
in \cite{DKSZ} that the classical Harnack inequality fails to hold if $N \geq 2$, even for global nonnegative solutions. On the other hand, it is known that
the Harnack inequality does hold for globally nonnegative local solutions in the purely time-dependent
situation (the case $N=0$) with $k$ given by \eqref{mainexample}, see \cite{Z0d}. Further, in the case
$N=1$, the Harnack inequality has been established for
{\em global} nonnegative {\em solutions} of \eqref{timefracdiff}
that are given as a spatial convolution of the initial data and the fundamental solution to \eqref{timefracdiff}, see \cite[Theorem 2.7]{DKSZ}. However, the question whether globally nonnegative {\em local solutions} to \eqref{timefracdiff} with $N=1$ satisfy the Harnack inequality remained an open problem. In this paper, we will answer this question in the affirmative, even in the much
more general case of globally nonnegative local weak solutions to \eqref{MProb}.

These results indicate that the space dimension has a crucial influence on the validity of the Harnack estimate.
A similar critical dimension phenomenon has been observed in \cite{KSVZ} in the context of temporal decay estimates for solutions of the time-fractional and more general subdiffusion equations in $\iR^N$. Roughly speaking, for small dimensions, the decay behaviour is similar to the one for the classical heat equation, with increasing decay rate for higher dimensions. Then, above a certain dimension, the decay rate no longer changes with increasing dimension, in stark contrast to the classical parabolic case.

Before describing our main result in more detail, we want to mention some other earlier results on Harnack type and H\"older estimates for non-local in time subdiffusion equations.   The first De Giorgi-Nash-Moser type results were obtained in \cite{base} (weak Harnack inequality) and \cite{Zhol} (H\"older continuity) for subdiffusion equations of the
form \eqref{MProb} in the case of fractional
time dynamics, i.e.\ with $k$ given by
\eqref{mainexample}. Concerning H\"older regularity results for problems involving single-order fractional derivative in time as well as space nonlocality we refer to \cite{Caf} and \cite{MXZ}. A weak Harnack inquality for such problems was obtained in~\cite{JPY}. For subdiffusion
equations \eqref{MProb} with a distributed order fractional time derivative, weak Harnack and H\"older estimates were recently established in \cite{harnackdistr}.

The main result of this paper is a (full) Harnack inequality for globally nonnegative local weak solutions to \eqref{MProb} in the one-dimensional case, which in particular fills the gap in the aforementioned results on Harnack estimates for the time-fractional diffusion equation \eqref{timefracdiff}. Our strategy of proof is much inspired by ideas from \cite{Kass1} on the spatially non-local equation \ref{space}. The proof consists of establishing a suitable local sup-estimate for nonnegative subsolutions to \eqref{MProb} and an improved weak Harnack inequality for weak supersolutions. In comparison with \cite{Kass1}, instead of a spatial tail term we have to study a history term and an initial data term. Here, we work in the same setting as in \cite{krz}, where the weak Harnack inequality has been proven for \eqref{MProb}
with a general kernel. Some assumptions on the kernel
$k$ are somewhat different from \cite{krz}, but the
main examples from \cite{krz} are all admissible here, too.

We impose the following assumptions on the kernel~$k$:\vspace{0.2 cm}

{\em 1. $\mathscr{PC}$ property, regularity, monotonicity and convexity:}
\eqq{
 k,l \in L_{1,\,loc}(\iR_+)\cap C^{1}((0,\infty)), \hd  k,l \m{ are nonnegative and nonincreasing}, k \m{ is convex and } k*l=1.
\tag{K0}}{pc}

{\em 2. Higher integrability of $l$ and comparability with the average:} There exist $p_{0}>1$, $t_0 > 0$ and $\overline{c} > 1$ such that
\eqq{l \in L_{p_0}((0,t_0)) \hd \m{ and }\;
\frac{1}{t} \int_{0}^{t}\big(l(s)\big)^{p_0} ds \leq \overline{c} \big(l(t)\big)^{p_0},\quad 0<t\le t_0.\tag{K1}}{ak1}

{\em 3. Upper estimate for $k$ via its derivative:} There exist $\tilde{c}\in (0,1)$ and $\tilde{t}_0 > 0$ such that
\eqq{
- t\dot{k}(t) \geq \tilde{c} {k(t)},\quad t \in (0,\tilde{t}_0). \tag{K2}
}{ak2}

{\em 4. Upper estimates for the derivative of $k$:}
There exist positive constants $c_0,\omega > 0$ and a
nondecreasing function $c : (0,\infty) \rightarrow (0,\infty)$ such that for all $M > 0$
\eqq{-\dot{k}(y) \leq  \frac{c_0}{y(1*l)(y)} \hd \m{ and } -\dot{k}(xy) \leq c(M)y^{-\omega}[-\dot{k}(x)], \quad 0<y \leq 1, \hd 0< x(1-y) \leq M. \hd \tag{K3}}{asholder}

In comparison to \eqref{pc} in \cite{krz}, we assume in addition that $k$ is convex, a property which,
to the best of our knowledge, is satisfied by all known examples of $\mathscr{PC}$ kernels. Notice that the first part of the assumption (\ref{asholder}) is weaker than assumption (K3) in \cite{krz}.

Let us now introduce the basic assumptions imposed on $A$ and $u_0$. Recall that we consider the case $N=1$. Denoting $\Omega_T=(0,T)\times \Omega$ with $T>0$ and a bounded open interval $\Omega \subset \mathbb{R}$ we assume that
\begin{itemize}
\item [{\bf (H1)}] $A\in L_\infty(\Omega_T)$. Set $\Lambda:=\esssup_{\Omega_T}|A|$.
\item [{\bf (H2)}] There exists $\nu>0$ such that
$A(t,x) \ge \nu$ for a.a.\ $(t,x)\in\Omega_T$.
\item [{\bf (H3)}] $u_0\in L_2(\Omega)$.
\end{itemize}

A function $u$ is a {\em weak solution (subsolution,
supersolution)} of (\ref{MProb}) in $\Omega_T$ if $u$ belongs to the space
\[
Z:=\{v \in L_{2}((0,T);H^1_2(\Omega)):\, k*v \in C([0,T];L_{2}(\Omega)), (k*v)|_{t=0} = 0\}
\]
and for any nonnegative test function
\[
\eta\in \oH^{1,1}_2(\Omega_T):=H^1_2((0,T);L_2(\Omega))\cap
L_2((0,T);\oH^1_2(\Omega)) \quad\quad
\Big(\oH^1_2(\Omega):=\overline{C_0^\infty(\Omega)}\,{}^{H^1_2(\Omega)}\Big)
\]
with $\eta|_{t=T}=0$ there holds
\begin{equation} \label{BWF}
\int_{0}^{T} \int_\Omega \Big(-\eta_t [k\ast (u-u_0)]+
(ADu|D \eta)\Big)\,dxdt= \,(\le,\,\ge )\, 0.
\end{equation}
\begin{bemerk1} \label{notationR}
{\em
Note that throughout the paper, we use the notation with
the scalar product $(\cdot|\cdot)$ in $\iR^N$ for the elliptic part even if $N=1$. This is intended to make comments on the validity of certain estimates in the general case of $N\in \iN$ (with $\Omega$ being a bounded domain of $\iR^N$) more comprehensible. For general $N\in \iN$,
$A\in L_\infty(\Omega_T;\iR^{N\times N})$ with $\Lambda=\norm{A}_\infty$
and the ellipticity condition (H2) takes the following form: There exists $\nu>0$ such that
$(A(t,x)\xi|\xi)\ge \nu|\xi|^2$ for a.a.\ $(t,x)\in \Omega_T$ and all $\xi\in \iR^N$.
}
\end{bemerk1}
To formulate the main result we introduce the function
\eqq{
\ki(t) = \frac{1}{(1*l)(t)},\quad t>0,
}{kidef}
and set
\eqq{
r_0 = \left(\izi l(t)dt \right)^{\frac{1}{2}}.
}{r0}
Note that if $l$ is not integrable on $\mathbb{R}_{+}$, then $r_0 = \infty$. In what follows, writing $\frac{1}{r_0}$,  we use the convention $\frac{1}{\infty} = 0$.

In contrast to the case of fractional
time dynamics, i.e.\ $k$ is given by
\eqref{mainexample}, the problem with general kernel $k$
lacks a scaling property, which makes local regularity estimates based on De Giorgi or Moser iterations considerably more difficult to obtain.
This fundamental problem has been solved in \cite{krz} (see also \cite{harnackdistr}) by identifying a suitable function $\Phi(r)$ describing (up to a constant) the length of the time interval of a local time-space cylinder
given that the radius of the spatial ball scales like $r$.
\begin{lemma}[\cite{krz}]\label{fi}
There exists a unique strictly increasing function $\Phi \in C([0,r_0))\cap C^{1}((0,r_0))$, where $r_0$ is defined in (\ref{r0}), such that $\Phi(0) = 0$, $\lim\limits_{r\rightarrow r_0}\Phi(r) = \infty$ and
\eqq{
 \kjr = r^{-2}.
}{zn1}
\end{lemma}
Let $\Phi$ be the function from Lemma~\ref{fi}. For $\delta\in(0,1)$, $t_0\ge 0$,
$\tau>0$, $r \in (0,r_0)$ and a ball $B(x_0,r)$, define the boxes
\eqnsl{
Q_-(t_0,x_0,r, \delta)&=(t_0,t_0+\delta\tau \vdr)\times B(x_0,\delta r),\\
Q_+(t_0,x_0,r,\delta)&=(t_0+(2-\delta)\tau \vdr,t_0+2\tau
\vdr)\times B(x_0,\delta r).
}{defQpm}
Note that we keep the notation $B(x_0,r)$ (or $B_{r}(x_0)$) for balls in one dimension as well as $Du$ for $u_x$, cf.\ Remark \ref{notationR}.

Now we are ready to present the main result of this paper.

\begin{satz} \label{strongHarnack}
Let $T>0$ and $\Omega\subset \iR$ be a bounded open
interval. Suppose the assumptions (H1)--(H3) and (\ref{pc})- (\ref{asholder}) are satisfied.
Then there exists $r^{*} \in (0,r_0)$ and  $\tau^{*} \in (0,1)$, $\tau^{*}=\tau^{*}(\nu,\Lambda,\overline{c})$ such that for every $\tau \in (0,\tau^{*})$ and for any $\delta \in (0,1)$, $r\in (0,r^*]$,  with $t_0-\delta\tau\vdr \geq 0$, $t_{0}+2\tau \vdr \leq T$ and any ball $B(x_{0},r) \subset \Omega$ and any positive weak
solution $u$ of (\ref{MProb}) in $(0,t_0+2\tau
\vdr)\times B(x_0, r)$ with $u_0\ge 0$ in $B(x_0,r)$
there holds
\[
\esup_{Q_-(t_0,x_0,r, \delta)} u
\le C \einf_{Q_+(t_0,x_0,r, \delta)} u ,
\]
where $C=C(\nu,\Lambda,\delta,\tau,\overline{c},\tilde{c},c_0,\omega)$.
\end{satz}
\begin{bemerk1}
{\em Note that in Theorem \ref{strongHarnack} we assume that $u$ is a weak solution {\em locally in space} and {\em globally to the left in time}. The global solution property in time can be replaced by a local
one. In fact, on a formal level our estimates only need the solution
property on a time-space cylinder which is local both in time and space. To make these estimates rigorous
would require further technicalities such as a
time localized regularization of the weak formulation. This
issue has already been discussed in earlier work, see e.g.\ \cite[page 102]{Zhol}. For the sake of convenience, we confine ourselves to consider
only weak solutions that are global to the left in time.
}
\end{bemerk1}

Theorem \ref{strongHarnack} applies, among others, to the following pairs of kernels:
\begin{itemize}
\item the time fractional derivative with exponential weight
\[
k(t) = \frac{t^{-\al}}{\Gamma(1-\al)}e^{-\gamma t}, \hd l(t) = \frac{t^{\al-1}}{\Gamma(\al)}e^{-\gamma t} + \gamma \izt e^{-\gamma \tau}\frac{\tau^{\al-1}}{\Gamma(\al)}d\tau, \hd \hd \gamma \geq 0, \hd \al \in (0,1),
\]
\item the distributed order fractional derivative
\[
k(t) = \izj \frac{t^{-\al}}{\Gamma(1-\al)} d\mu(\al), \hd l(t) = \frac{1}{\pi}\izi e^{-pt}\frac{\izj p^{\al} \sin (\pi \al)\dd}{(\izj p^{\al} \sin (\pi \al)\dd)^{2} + (\izj p^{\al} \cos (\pi \al)\dd)^{2}}dp,
\]
where $\mu$ is a finite sum $d\mu=\sum_{n=1}^M q_nd\delta(\cdot-\alpha_n)+wd\lambda$, $\alpha_n\in (0,1)$, $q_n\ge 0$ for all $n=1,\ldots,M$,
$w\in L_1((0,1))$ is nonnegative and $\mu \not \equiv 0$ ($\delta(\cdot - \al_{n})$ is the Dirac measure at $\alpha_n$ and $\lambda$ denotes the one-dimensional Lebesgue measure),
\item the fractional derivative with $l$ decaying exponentially
\[
k(t) =\frac{t^{\al-1}}{\Gamma(\al)}e^{-\gamma t} + \gamma \izt e^{-\gamma \tau}\frac{\tau^{\al-1}}{\Gamma(\al)}d\tau, \hd l(t) = \frac{t^{-\al}}{\Gamma(1-\al)}e^{-\gamma t}, \hd \hd \gamma \geq 0, \hd \al \in (0,1),
\]
\item the distributed order case with switched kernels under the assumption $\supp \mu \subset [0,\al_{*}], \al_{*} \in (0,1)$
\[
k(t) = \frac{1}{\pi}\izi e^{-pt}\frac{\izj p^{\al} \sin (\pi \al)\dd}{(\izj p^{\al} \sin (\pi \al)\dd)^{2} + (\izj p^{\al} \cos (\pi \al)\dd)^{2}}dp, \hd l(t) = \izj \frac{t^{-\al}}{\Gamma(1-\al)} d\mu(\al).
\]
\end{itemize}
For a proof that the examples listed above satisfy the assumptions (\ref{pc}) - (\ref{asholder}) we refer to \cite{krz}.

The paper is organised as follows. In Section 2 we recall some auxiliary results on Moser iterations, parabolic embeddings and properties of $\mathscr{PC}$ kernels that follow from assumption (\ref{ak1}). Section 3 is devoted to the proof of an improved weak Harnack inequality in one dimension. In Section 4 we establish local $\sup$-estimates for subsolutions. The main result, Theorem \ref{strongHarnack}, follows directly from Theorem \ref{imweakHarnack} and Theorem \ref{subestimate}.




\section{ Preliminaries}
\subsection{Iterations and embedding}
We begin this preliminary section with a simple but important iteration lemma from \cite{Giaquinta}.
\begin{lemma}\cite[Lemma V.3.1]{Giaquinta}\label{geometric}
Assume that $f$ is a nonnegative,
real-valued, bounded function defined on an interval $[\tau_1,\tau_2] \subset \mathbb{R}_{+}$. Assume
further that for all $ \tau_1\leq  s < t \leq \tau_2$ we have
\[
f(s) \leq \theta f(t) + A_1(t-s)^{-\al} + A_2(t-s)^{-\beta} + B
\]
for some nonnegative constants $A,B,\al,\beta$ and $\theta \in [0,1)$. Then,
\[
f(\tau_1) \leq c(\al,\theta)[A_1(\tau_2-\tau_1)^{-\al}+A_2(\tau_2-\tau_1)^{-\beta}+B].
\]
\end{lemma}
In the estimates for subsolutions we apply Moser's iteration technique in spaces of functions with mixed integrability in time and space. We therefore introduce the following notation. Let
$U_\sigma = U_\sigma^{t} \times U_\sigma^{x}$, $0<\sigma\le 1$, be a collection
of measurable subsets of the measure space $U_1$ endowed with the product measure
$\mu=\mu_t \otimes \mu_x$ such that $U_{\sigma'}^t\subset
U_\sigma^t$ and $U_{\sigma'}^x\subset
U_\sigma^x$ if $\sigma'\le \sigma$. For $p\in (0,\infty)$, $q \in (0,\infty]$ and
$0<\sigma\le 1$, by
$L_{p,q}(U_\sigma,d\mu)$ we mean the space of all $\mu$-measurable functions
$f:U_\sigma \rightarrow \iR$ for which

\[
\|f\|_{L_{p,q}(U_\sigma)}:=\left(\int_{U_\sigma^{t}}(\int_{U_\sigma^x}|f|^q\,d\mu_x)^{p/q}d\mu_t\right)^{\frac{1}{p}}<\infty \hd \m{ if } \hd q<\infty \hd
\]
and
\[
\|f\|_{L_{p,\infty}(U_\sigma)}:=\left(\int_{U_\sigma^{t}}(\esssup_{x \in U_{\sigma}^{x}}|f|)^{p}d\mu_t\right)^{\frac{1}{p}}<\infty.
\]

Below we present a Moser iteration lemma which may be proven analogously to \cite{base}[Lemma 2.1] and \cite{harnackdistr}[Lemma 2.7, Lemma 2.8], see also \cite{CZ}, \cite{AS}.

\begin{lemma} \label{moserit1}
Let $\kappa>1$, $\beta_1,\beta_2 \geq 1$, $\bar{p}\ge 1$, $C\ge 1$, $\delta \in (0,1)$ and $a>0$. Suppose
$f$ is a $\mu$-measurable function on $U_1$ such that
\begin{equation} \label{mositer1}
\|f\|_{L_{\beta_1 \gamma\kappa,\beta_2\gamma\kappa}(U_{\sigma'})}\le
\Big(\frac{C(1+\gamma)^{a}}{(\sigma-\sigma')^{a}}\Big)^{1/\gamma}\,\|f\|_{L_{\beta_1 \gamma, \beta_2 \gamma}(U_{\sigma})},\quad \mbox{for every }
\; \delta<\sigma'<\sigma\le 1,\;\gamma>0.
\end{equation}
Then there exists a constant $M=M(a,\kappa,\bar{p})$  such that
\[
\esup_{U_{\delta}}{|f|} \le
\Big(\frac{M C^{\frac{\kappa}{\kappa-1}}}{(1-\delta)^{\gamma_0}}\Big)^{1/p}
\|f\|_{L_{\beta_1 p,\beta_2 p}(U_1)}\quad \mbox{for all} \;p\in
(0,\bar{p}] ,\]
where $\gamma_{0}= \frac{a \kappa }{\kappa -1}$.
\end{lemma}


$\mbox{}$

Finally we recall a parabolic embedding result in one space dimension. It follows from the Gagliardo-Nirenberg and H\"older's inequality, cf.\ for example, \cite{krz}, \cite[Section 2]{Za2}.

\begin{prop}
Let  $T>0$ and $\Omega$ be a bounded interval in $\iR$. For $1<p\le
\infty$ we define the space
\begin{equation} \label{Vdef}
V_p:=V_p((0,T)\times \Omega)=L_{2p}((0,T);L_2(\Omega))\cap
L_2((0,T);H^1_2(\Omega)),
\end{equation}
endowed with the norm
\[
\|u\|_{V_p((0,T)\times \Omega)}:=\|u\|_{L_{2p}((0,T);L_2(\Omega))}
+\|Du\|_{L_2((0,T);L_2(\Omega))}.
\]

Then, if
\[
p'\left(1-\frac{2}{a}\right) -\frac{1}{b} = \frac{1}{2},
\]
where $p'=\frac{p}{p-1}$ and
\[
a\in \left[ \frac{4p}{p+1},2p \right], \hd b\in[2, \infty],
\]
then
$V_p\hookrightarrow
L_{a}((0,T); L_{b}(\Omega))$ and there is $C=C(b)$ such that for all
$u\in V_{p}\cap L_{2}((0,T);\oH^1_2 (\Om))$
\eqq{
\norm{u}_{L_{a}((0,T); L_{b}(\Omega))} \leq C\norm{u}_{L_{2p}((0,T); L_{2}(\Omega))}^{1-\theta}\norm{D u}_{L_{2}((0,T); L_{2}(\Omega))}^{\theta} 
}{sobolev}
with $\theta = \frac{1}{2}-\frac{1}{b}$.
\end{prop}

\subsection{Properties of the kernels} \label{SecYos}
We begin this subsection by recalling two simple but useful lemmas, see \cite{harnackdistr}, \cite{krz}, and \cite{base} in case $l(t)=g_\alpha(t):=\frac{t^{\al-1}}{\Gamma(\al)}$. By ${}_0H^1_1([0,T])$ we mean the space of all $H^1_1([0,T])$-functions that vanish at $0$.
\begin{lemma} \label{comm}
Let $T>0$ and $l\in L_{1}((0,T))\cap C^1((0,T))$ be nonnegative and nonincreasing. Suppose that $v\in
{}_0H^1_1([0,T])$ and $\varphi\in C^1([0,T])$. Then
\[
\big(l\ast(\varphi \dot{v}))(t)=\varphi(t)(l\ast
\dot{v})(t)+\int_0^t
v(\sigma)\partial_\sigma\big(l(t-\sigma)
[\varphi(t)-\varphi(\sigma)]\big)\,d\sigma,\;\;\mbox{a.a.}\;t\in
(0,T).
\]
If in addition $v$ is nonnegative and $\varphi$ is nondecreasing
there holds
\[
\big(l\ast(\varphi \dot{v}))(t)\ge \varphi(t)(l\ast
\dot{v})(t)-\int_0^t l(t-\sigma)
\dot{\varphi}(\sigma)v(\sigma)\,d\sigma,\;\;\mbox{a.a.}\;t\in
(0,T).
\]
\end{lemma}

\begin{lemma} \label{comm2}
Let $T>0$, $k\in H^1_1([0,T])$, $v\in L_1([0,T])$, and $\varphi\in
C^1([0,T])$. Then
\[
\varphi(t)\,\frac{d}{dt}\,(k\ast v)(t)=\,\frac{d}{dt}\,\big(k\ast
[\varphi v]\big)(t)+\int_0^t
\dot{k}(t-\tau)\big(\varphi(t)-\varphi(\tau)\big)v(\tau)\,d\tau,\;\;\mbox{a.a.}\;t\in
(0,T).
\]
\end{lemma}
Let us now recall important properties of the kernels $k,l,k_1$ under the assumptions \eqref{pc} and \eqref{ak1}. For the proofs we refer to \cite[Section 2]{krz}.

\begin{lemma}\label{kernels}
Assume \eqref{pc} and \eqref{ak1}. Then
\[
k(t) \leq  k_1(t),\quad t>0,
\]
and
\[
(1*k)(t) \leq \overline{c} t k_1(t),\quad t \in (0,t_0),
\]
where $t_0$ and $\overline{c}$ come from the assumption (\ref{ak1}).
\end{lemma}

\begin{lemma}\label{ba}
Assume \eqref{pc} and \eqref{ak1}. Then
\[
k_1(xy) \leq \max\{1,y^{-1}\}k_1(x),\quad x,y>0.
\]
\end{lemma}

\begin{lemma}\label{scaling}
Assume \eqref{pc} and \eqref{ak1}. Then
there exists $r^{*} \in (0,r_0)$ such that $\Phi(r^{*})~\leq~\min\{1,t_{0}\}$ and for every  $1 \leq  p \leq p_0$ and every $r \in (0,r^{*}]$, there holds
\[
\norm{l}_{L_{p}(0,\vdr)}^{p}(\vdr)^{p-1} \leq C r^{2p},
\]
where $C>0$ is a positive constant depending only on $p_0,\bar{c}$ from the assumption (\ref{ak1}).
\end{lemma}
\subsection{Regularised weak formulation and time-shifts}
In order to derive {\em a priori} estimates for (\ref{MProb}) it is convenient to have a suitable time regularised weak formulation of (\ref{MProb}) at hand. To this end, following previous works (see e.g.\ \cite{krz,Zhol,Za2}), we use the Yosida approximation of the  operator $\frac{d}{dt}(k*\cdot)$, cf.\ also \cite{Grip1}.
Let  $1\le p<\infty$, $T>0$, and $X$ be a real Banach
space. Then the operator $B$ defined by
\[ B u=\,\frac{d}{dt}\,(k\ast u),\;\;D(B)=\{u\in L_p([0,T];X):\,k\ast u\in \mbox{}_0 H^1_p([0,T];X)\},
\]
where the zero means vanishing trace at $t=0$, is
$m$-accretive in $L_p([0,T];X)$, see (\cite{Phil1}, \cite{CP},
\cite{Grip1}).
There exist nonnegative functions $h_{n}\in L_{1,\,loc}(\iR_+)$, $n \in \mathbb{N}$, such that
if $f\in L_p([0,T];X)$, $1\le
p<\infty$, then $h_{n}\ast f\to f$ in $L_p([0,T];X)$
as $n\to \infty$, and $k_n:=k*h_n$ belongs to $H^1_1([0,T])$ and we have
 $ \frac{d}{dt}(k_n*u)\rightarrow \frac{d}{dt}(k*u)$ in $L_p([0,T];X)$ as
$n\to \infty$ for every $u\in D(B)$, see e.g.\ \cite[Section 2]{Za2}.

By means of the Yosida approximation of $B$, given by $\frac{d}{dt}(k_n*\cdot)$, one may replace the singular
kernel $k$ by the more regular kernel $k_{n}$ ($n\in\iN$) in the weak formulation of (\ref{MProb}). This is the subject of the following lemma, which we formulate here for the case $N=1$, although it also applies to the multidimensional case.
\begin{lemma}\cite[Lemma 3.1]{Za2} \label{LemmaReg}
Let  $T>0$ and $\Omega\subset \iR$ be a bounded open
interval. Suppose the assumptions (H1)--(H3) are satisfied. Then $u\in
Z$ is a weak solution (subsolution, supersolution) of
(\ref{MProb}) in $\Omega_T$ if and only if for any nonnegative
function $\psi\in \oH^1_2(\Omega)$ one has
\eqq{
\int_\Omega \Big(\psi \partial_t[k_{n}\ast
(u-u_0)]+(h_n\ast [ADu]|D\psi)\Big)\,dx \\
=\,(\le,\,\ge)\,0 ,\quad\mbox{a.a.}\;t\in (0,T),\,n\in \iN.
}{LemmaRegF}
\end{lemma}

When establishing local estimates for problems with memory, it is often useful to apply time-shifts.
Let $t_1\in (0,T)$ be fixed and suppose that $t\in (t_1,T)$. We shift the time by introducing a new time
$s=t-t_1$ and we set $\tilde{g}(s)=g(s+t_1)$, $s\in
(0,T-t_1)$, for functions $g$ defined on $(t_1,T)$.
In all the previous works \cite{harnackdistr,krz,base,Zhol}
on local estimates for non-local in time subdiffusions,
time shifts were applied to the regularised weak formulation given by (the multi-dimensional version of)
Lemma \ref{LemmaReg}. The idea is to decompose the
convolution term $k_n\ast u$ as
\begin{equation} \label{decompmem}
(k_n\ast u)(t,x)=\int_{t_1}^t
k_n(t-\tau)u(\tau,x)\,d\tau+\int_{0}^{t_1}
k_n(t-\tau)u(\tau,x)\,d\tau,\quad t\in (t_1,T).
\end{equation}
Together with the time derivative, this leads to
\[
\partial_t(k_n\ast u)(t,x)=\partial_s(k_n\ast \tilde{u})(s,x)+
\int_{0}^{t_1}
\dot{k}_n(s+t_1-\tau)u(\tau,x)\,d\tau,
\]
that is, a sum of the shifted (regularised) integro-differential operator in time and a memory term. The regularisation of the kernel is crucial for the first term in order to obtain energy estimates in a rigorous way
by using the key identity \eqref{fundidentity} from below
and sending $n\to \infty$ afterwards. It turns out that for our proof
of the full Harnack inequality it is better to first shift the time and then to regularise the kernel. The reason for this is that in contrast to \cite{harnackdistr,krz,base,Zhol}, we do not estimate $u$
in the memory term (e.g.\ by using positivity or an assumed upper bound), but keep the memory term as it is with the kernel $k$.

In order to get a suitable regularised weak formulation with time-shift we
use the formulation from Lemma \ref{LemmaReg} as our starting point and proceed as follows.
For arbitrary $t_1\in (0,T)$, multiply (\ref{LemmaRegF}) by a nonnegative $\Psi \in C^{1}([0,T])$ with $\Psi(T) = 0$ and integrate in time from $t_1$ to $T$. Then
\eqq{
\int_{t_1}^{T}\Psi\int_\Omega \Big(\psi \partial_t[k_{n}\ast
(u-u_0)]+(h_n\ast [ADu]|D\psi)\Big)\,dxdt \\
=\,(\le,\,\ge)\,0.
}{jgjg}
Integration by parts in time gives
\[
\int_{t_1}^{T}\Psi\int_\Omega \psi \partial_t (k_{n}\ast
u) dxdt = -\Psi(t_1)\int_\Omega \psi (k_{n}\ast
u)(t_1,\cdot)dx - \int_{t_1}^{T}\dot{\Psi}\int_\Omega \psi (k_{n}\ast
u) dxdt.
\]
We insert this identity in (\ref{jgjg}) and pass to the limit with $n$ (choosing an appropriate subsequence if necessary) to obtain
\[
-\Psi(t_1)\int_\Omega \psi (k\ast
u)(t_1,\cdot)dx - \int_{t_1}^{T}\dot{\Psi}\int_\Omega \psi (k\ast
u)\, dxdt + \int_{t_1}^{T}\Psi\int_\Omega ([ADu]|D\psi)\, dx dt
\]
\[
=\,(\le,\,\ge)\, \int_\Omega \psi u_0\, dx \int_{t_1}^{T} \Psi \cdot k\,dt, \quad t_{1} \in (0,T).
\]
We now shift the time as described above (with $s=t-t_1$ and $\tilde{g}(s)=g(s+t_1)$, $s\in
(0,T-t_1)$, for functions $g$ defined on $(t_1,T)$). From the
decomposition
\[
(k\ast u)(t,x)=\int_{t_1}^t
k(t-\tau)u(\tau,x)\,d\tau+\int_{0}^{t_1}
k(t-\tau)u(\tau,x)\,d\tau,\quad t\in (t_1,T),
\]
we then deduce that
\[
-\Psi(t_1)\int_\Omega \psi (k\ast
u)(t_1,\cdot)dx - \int_{0}^{T-t_1}\dot{\tilde{{\Psi}}}\int_\Omega \psi (k\ast
\tilde{u}) dxds -  \int_{0}^{T-t_1}\dot{\tilde{{\Psi}}}\int_\Omega \psi \int_{0}^{t_1} k(s+t_1-\xi)u(\xi,\cdot)d\xi dxds
\]
\eqq{
+\int_{0}^{T-t_1}\tilde{\Psi}\int_\Omega (\widetilde{[ADu]}|D\psi)\, dx ds =\,(\le,\,\ge)\, \int_\Omega \psi u_0\, dx \int_{0}^{T-t_1}\tilde{\Psi}(s) k(s+t_1)\,ds,\quad  t_{1} \in (0,T).
}{jjhjh}
We note that since $u$ is a nonnegative weak (sub/super) solution we have $k*u \in C([0,T];L_{2}(\Omega))$ and since $k$ is convex
\[
\abs{\int_{0}^{t_1} \dot{k}(s+t_1-\xi)u(\xi,x)d\xi} \leq (k*u)(t_1,x)\quad  \m{ a.e. in } \Omega.
\]
Thus, we may integrate by parts the third term on the left-hand-side of (\ref{jjhjh}) to the result
\[
 - \int_{0}^{T-t_1}\dot{\tilde{{\Psi}}}\int_\Omega \psi (k\ast
\tilde{u})(s)\, dxds +  \int_{0}^{T-t_1}\tilde{{\Psi}}\int_\Omega \psi \int_{0}^{t_1} \dot{k}(s+t_1-\xi)u(\xi,\cdot)\,d\xi\, dx ds
\]
\[
+\int_{0}^{T-t_1}\tilde{\Psi}\int_\Omega ([\tilde{A}D\tilde{u}]|D\psi)\, dx ds =\,(\le,\,\ge)\, \int_\Omega \psi u_0 dx \int_{0}^{T-t_1}\tilde{\Psi}(s) k(s+t_1)ds.
\]

We now proceed as in the proof of \cite[Lemma 3.1]{Za2}. We take $\tilde{\Psi}(s) = \int_{s}^{T-t_1}h_n(\sigma-s)\varphi(\sigma)d\sigma$ with nonnegative $\varphi \in C^{1}([0,T-t_1])$ satisfying $\varphi(T-t_1) = 0$ and arbitrary $n \in \mathbb{N}$. Then, by Fubini's theorem
\[
 - \int_{0}^{T-t_1}\dot{\varphi}\int_\Omega \psi (k_n\ast
\tilde{u})\, dxds
+\int_{0}^{T-t_1}\varphi\int_\Omega (h_n*[\tilde{A}D\tilde{u}]|D\psi)\, dx ds =\,(\le,\,\ge)\,
\]
\[
\int_{0}^{T-t_1}\varphi \int_\Omega \psi\, h_n*\left(\int_{0}^{t_1} [-\dot{k}(\cdot+t_1-\xi)]u(\xi,\cdot)d\xi\right) dx ds +
 \int_\Omega \psi u_0 dx \int_{0}^{T-t_1}\varphi (h_n* k(\cdot+t_1))ds.
\]
Applying first integration by parts and then using the fact that $\varphi$ is arbitrary, we arrive at
\[
\int_\Omega \psi \partial_s (k_n\ast
\tilde{u}) dx
+\int_\Omega (h_n*[\tilde{A}D\tilde{u}]|D\psi)\, dx =\,(\le,\,\ge)\,
\]
\eqq{
 \int_\Omega \psi h_n*\left(\int_{0}^{t_1} [-\dot{k}(\cdot+t_1-\xi)]u(\xi,\cdot)d\xi\right) dx  +
 \int_\Omega \psi u_0 \,dx \, (h_n* k(\cdot+t_1)),\;\text{a.a.}\;s\in (0,T-t_1),
}{shiftprob}
for any nonnegative function $\psi\in \oH^1_2(\Omega)$ and all $t_1 \in (0,T)$.  Formula (\ref{shiftprob}) will be the starting point for all the estimates in this paper.

We finish this preliminary section recalling the ``fundamental identity" for integro-differential
operators of the from $\frac{d}{dt}(k\ast \cdot)$ (see also \cite{Za2},\cite{GLS}).
Suppose $k\in H^1_1([0,T])$ and $H\in C^1(\iR)$. Then it follows
from a straightforward computation that for any sufficiently smooth
function $u$ on $(0,T)$ one has for a.a. $t\in (0,T)$,
\begin{align} \label{fundidentity}
H'(u(t))&\frac{d}{dt}\,(k \ast u)(t) =\;\frac{d}{dt}\,(k\ast
H(u))(t)+
\Big(-H(u(t))+H'(u(t))u(t)\Big)k(t) \nonumber\\
 & +\int_0^t
\Big(H(u(t-s))-H(u(t))-H'(u(t))[u(t-s)-u(t)]\Big)[-\dot{k}(s)]\,ds.
\end{align}

This identity plays a crucial role in all the estimates below.
Since the identity requires particular regularity of the kernels, it is often applied to the approximation $k_n$.\\

\section{The improved weak Harnack inequality}
\begin{satz}\label{imweakHarnack}
Let  $\Omega\subset \iR$ be a bounded domain.
 Suppose the assumptions (H1)--(H3) and (\ref{pc})-- (\ref{ak2}) are satisfied.
 There exist   $r^{*} \in (0,r_0)$ and  $\eta^{*} \in (0,1)$, $ \eta^{*}=\eta^{*}(\nu,\Lambda,\overline{c}) > 0$ such that for every $\eta \in (0,\eta^{*})$ and any  $t_0\ge 0$, $r\in (0,r^*]$,  $\delta\in(0,1)$  with $t_{0}+2\eta \vdr \leq T$ and any ball $B(x_{0},r) \subset \Omega$ and any weak
supersolution $u \geq \ve >0$ of (\ref{MProb}) in $(0,t_0+2\eta
\vdr)\times B(x_0, r)$ with $u_0\ge 0$ in $B(x_0,
r)$ there holds

\[
 r\int_{0}^{t_{0}}[-\dot{k}(t_0+\sigma'\eta\vdr-\tau)]\int_{B_{r\sigma'}(x_0)}u(\tau,x)\,dx\,d\tau +r k(\eta\vdr+t_0) \int_{B_{r\sigma'}(x_0)} u_{0}\,dx
\]
\[
+ \frac{1}{r\vdr}\int_{t_0}^{t_0 + \sigma\eta\vdr}\int_{B_{r \sigma}(x_0)} u\, dx\,dt  \leq C (\sigma-\sigma')^{-4}  \inf_{(t_0+(2-\sigma)\eta\vdr,t_0+2\eta\vdr)\times B_{r\sigma }(x_0)} u,
\]
for any $0<\delta \leq \sigma'<\sigma<1$, where $C$ is a positive constant that depends only on $\Lambda, \nu,\eta, \delta, \sigma, \overline{c},\tilde{c}, p_0$.
\end{satz}
\begin{proof}
Due to the weak Harnack estimate \cite[Theorem 1.1]{krz} it is enough to estimate the history term and the $u_0$ - term by the infimum of the supersolution.
In the following we abbreviate the notation by setting $B_{r} := B_{r}(x_{0})$, because we only consider balls (intervals) centered at a fixed $x_0$. Let us fix $0<\delta\leq \sigma'<\sigma<1$ and  $r^{*} \in (0,r_0)$ such that the weak Harnack inequality \cite[Theorem 1.1]{krz} holds for $r \in (0,r^{*}/2)$, and Lemma \ref{scaling} holds for $r \in (0,r^{*})$. Moreover, for $r\in (0,r^{*})$, we set $t_1=t_0+\sigma'\eta\vdr$, $t_2 = t_0+\sigma\eta\vdr$. Then, we
shift the time by setting ${s}=t-t_0$ and  $\tilde{f}(s)=f(s+t_0)$, $s\in
(0,t_2-t_0)$, for functions $f$ defined on $(t_0,t_2)$. Since
 $u$ is a weak supersolution of
(\ref{MProb}) in $(0,t_2)\times B_r$, we have (cf. (\ref{shiftprob})) for a.a. $s\in
(0,t_2-t_0)$
\[
\int_{\Omega} \Big(v \partial_s(k_{n}\ast \tilde{u})+(h_n\ast
[\tilde{A}D\tilde{u}]|Dv\big)\Big)\,dx
\]
\begin{equation} \label{supp0}
\geq  \int_\Omega v\, h_n*\left(\int_{0}^{t_0} [-\dot{k}(\cdot+t_0-\tau)]u(\tau,\cdot)d\tau\right) dx  +
 \int_\Omega v u_0\, dx\,  (h_n* k(\cdot+t_0)),
\end{equation}
for any nonnegative function $v\in \oH^1_2(B_r)$.  We introduce the cut-off function $\psi\in C^1_0(B_{r})$ such that
$0\le \psi\le 1$, $\psi=1$ in $B_{r\sigma' }$, supp$\,\psi\subset B_{r\sigma }$, and $|D \psi|\le 2/[r(\sigma-\sigma')]$. Then, we choose  in (\ref{supp0}) the test function $v=\psi^2
\tilde{u}^{\beta}$ with $\beta \in (-1,0)$.
The fundamental identity gives
\[
\tilde{u}^{\beta}\partial_{s}(k_{n}\ast \tilde{u}) \leq
\,\frac{1}{1+\beta}\,\partial_{s}
(k_{n}\ast\tilde{u}^{1+\beta}),\quad
\mbox{a.a.}\;(s,x)\in (0,t_2-t_0)\times B_r.
\]
Thus we arrive at
\[
\frac{1}{1+\beta}\,\int_{B_{r}}\partial_{s}
(k_{n}\ast\tilde{u}^{1+\beta})\psi^{2}dx + \big(h_n\ast
[\tilde{A}D\tilde{u}]|D(\psi^2
\tilde{u}^{\beta})\big)\,dx
\]
\[
\geq  \int_{B_{r}} u_{0} \psi^2
\tilde{u}^{\beta} dx (h_n* k(\cdot+t_0)) +   \int_{B_{r}}h_n*\left(\int_{0}^{t_{0}}[-\dot{k}(\cdot+t_0-\tau)]u(\tau,x)d\tau \right)\tilde{u}^{\beta}\psi^{2}dx,
\]
which is equivalent to
\begin{align}
\,&\frac{1}{1+\beta}\, \int_{B_{r}}\partial_{s}
(k_{n}\ast\tilde{u}^{1+\beta})\psi^2\,dx+\beta\int_{B_{r}}\big(h_n\ast
[\tilde{A}D\tilde{u}]|\psi^2 \tilde{u}^{\beta-1}D
\tilde{u}\big)\,dx + \,2\int_{B_{r}}\big(h_n\ast [\tilde{A}D\tilde{u}]|\psi D\psi
\,\tilde{u}^{\beta}\big)\,dx\nonumber\\
  &  \geq \int_{B_{r}} u_{0} \psi^2
\tilde{u}^{\beta} dx (h_n* k(\cdot+t_0)) +   \int_{B_{r}}h_n*\left(\int_{0}^{t_{0}}[-\dot{k}(\cdot+t_0-\tau)]u(\tau,x)d\tau\right) \tilde{u}^{\beta}\psi^{2} dx. \label{L1}
\end{align}

Next, choose $\varphi\in C^1([0,\sigma\eta\vdr])$ such that $0\le
\varphi\le 1$, $\varphi=1$ in $[0,\sigma'\eta\vdr]$, $\varphi=0$ in
$[\frac{\sigma'+\sigma}{2}\eta\vdr,\sigma\eta\vdr]$, and $0\le -\dot{\varphi}\le
\frac{4}{(\sigma-\sigma')\eta\vdr}$. We multiply (\ref{L1}) by $(1+\beta)\varphi(s)$ (recall that $1+\beta > 0$) and apply Lemma~\ref{comm2} to the first term
to get
\begin{align*}
&\int_{B_{r}}
\partial_{s}(k_{n}\ast
[\varphi\psi^2\tilde{u}^{1+\beta}]\big)\,dx+\beta(1+\beta)\,
\int_{B_{r}}\big(\tilde{A}D\tilde{u}|\psi^2 \tilde{u}^{\beta-1}D
\tilde{u}\big)\varphi\,dx + 2(1+\beta)\,\int_{B_{r}}\big(\tilde{A}D\tilde{u}|\psi D\psi
\,\tilde{u}^{\beta}\big)\varphi\,dx \nonumber\\
&\geq (1+\beta)\varphi\left[\int_{B_{r}} u_{0} \psi^2
\tilde{u}^{\beta} dx  (h_n*k(\cdot+t_0)) + \int_{B_{r}}h_n*\left(\int_{0}^{t_{0}}[-\dot{k}(\cdot+t_0-\tau)]u(\tau,x)d\tau \right)\tilde{u}^{\beta}\psi^{2} dx\right]\nonumber\\
& \,\;\;-\int_0^s
\dot{k}_{n}(s-\tau)\big(\varphi(s)-\varphi(\tau)\big)
\big(\int_{B_{r}}\psi^2\tilde{u}^{1+\beta}\,dx\big)(\tau)\,d\tau  +  \mathcal{R}_n(s),
\end{align*}
where
\begin{align*}
\mathcal{R}_n(s)= &\,\,-\beta(1+\beta)\, \int_{B_{r}}\big(h_n\ast
[\tilde{A}D\tilde{u}]-\tilde{A}D\tilde{u}|\psi^2 \tilde{u}^{\beta-1}D
\tilde{u}\big)\varphi\,dx\\
&\,-2(1+\beta)\,\int_{B_{r}}\big(h_n\ast
[\tilde{A}D\tilde{u}]-\tilde{A}D\tilde{u}|\psi D\psi
\,\tilde{u}^{\beta}\big)\varphi\,dx.
\end{align*}
Since $\beta \in (-1,0)$, the assumptions (H1) and (H2) lead to
\begin{align*}
\int_{B_{r}}  &
\partial_{s}\big(k_{n}\ast
[\varphi\psi^2\tilde{u}^{1+\beta}]\big)\,dx+\frac{4\beta}{(1+\beta)}\nu
\int_{B_{r}}\abs{D\tilde{u}^{\frac{\beta+1}{2}}}^{2}\psi^2 \varphi\,dx + 2(1+\beta)\Lambda\,\int_{B_{r}}|D\tilde{u}|\psi |D\psi|
\,\tilde{u}^{\frac{\beta-1}{2}}\tilde{u}^{\frac{\beta+1}{2}}\varphi\,dx \nonumber \\ \geq & (1+\beta)\varphi\left[\int_{B_{r}} u_{0} \psi^2
\tilde{u}^{\beta} dx \, (h_n*k(\cdot+t_0)) + \int_{B_{r}}h_n*\left(\int_{0}^{t_{0}}[-\dot{k}(\cdot+t_0-\tau)]u(\tau,x)d\tau \right)\tilde{u}^{\beta}\psi^{2} dx\right]\nonumber\\
& \,-\int_0^s
\dot{k}_{n}(s-\tau)\big(\varphi(s)-\varphi(\tau)\big)
\big(\int_{B_{r}}\psi^2\tilde{u}^{1+\beta}\,dx\big)(\tau)\,d\tau  +\mathcal{R}_n(s) .
\end{align*}
Applying Young's inequality we obtain
\begin{align}\label{L121}
\int_{B_{r}}  &
\partial_{s}\big(k_{n}\ast
[\varphi\psi^2\tilde{u}^{1+\beta}]\big)\,dx+\frac{2\beta\nu}{(1+\beta)}
\int_{B_{r}}\abs{D\tilde{u}^{\frac{\beta+1}{2}}}^{2}\psi^2 \varphi\,dx + \frac{8\Lambda^2(\beta+1)}{|\beta| \nu}\,\int_{B_{r\sigma }}
\tilde{u}^{\beta+1}|D\psi|^2\varphi\,dx \nonumber \\
\geq & (1+\beta)\varphi\left[\int_{B_{r}} u_{0} \psi^2
\tilde{u}^{\beta} dx  (h_n*k(\cdot+t_0)) + \int_{B_{r}}h_n*\left(\int_{0}^{t_{0}}[-\dot{k}(\cdot+t_0-\tau)]u(\tau,x)d\tau \right)\tilde{u}^{\beta}\psi^{2} dx\right]\nonumber\\
& \,-\int_0^s
\dot{k}_{n}(s-\tau)\big(\varphi(s)-\varphi(\tau)\big)
\big(\int_{B_{r}}\psi^2\tilde{u}^{1+\beta}\,dx\big)(\tau)\,d\tau  +\mathcal{R}_n(s),\quad \text{a.a.}\;s\in (0,t_2-t_0).
\end{align}
We note that
\eqnsl{
-\int_0^{s}\!\!&\!\int_0^\xi
\dot{k}_{n}(\xi-\tau)\big(\varphi(\xi)-\varphi(\tau)\big)
\big(\int_{B_{r}}\psi^2\tilde{u}^{1+\beta}\,dx\big)(\tau)\,d\tau\,d\xi \nonumber\\
& \,=-\int_0^{s}
k_{n}(s-\tau)\big(\varphi(s)-\varphi(\tau)\big)
\big(\int_{B_{r}}\psi^2\tilde{u}^{1+\beta}\,dx\big)(\tau)\,d\tau\\
&\,+\int_0^{s}\!\!\!\dot{\varphi}(\xi)\int_0^\xi
k_{n}(\xi-\tau)
\big(\int_{B_{r}}\psi^2\tilde{u}^{1+\beta}\,dx\big)(\tau)\,d\tau\,d\xi\\
\geq &\,\int_0^{s}\!\!\!\dot{\varphi}(\xi)\int_0^\xi
k_{n}(\xi-\tau)
\big(\int_{B_{r}}\psi^2\tilde{u}^{1+\beta}\,dx\big)(\tau)\,d\tau\,d\xi,\quad s\in (0,t_2-t_0],
}{kndot}
because $\varphi$ is nonincreasing. Hence, if we integrate (\ref{L121}) in time,  drop the second term, apply the inequality above and finally pass to the limit with $n$ we arrive at
\begin{align*}
\int_{B_{r}}  &
(k*[\varphi\psi^2\tilde{u}^{1+\beta}])(s)\,dx+ \frac{8\Lambda^2 (\beta+1)}{|\beta| \nu}\,\int_{0}^{s}\int_{B_{r\sigma }}
\tilde{u}^{\beta+1}|D\psi|^2\varphi\,dxd\xi
\nonumber\\
- &\int_0^{s}\!\!\!\dot{\varphi}(\xi)\int_0^\xi
k(\xi-\tau)
\big(\int_{B_{r}}\psi^2\tilde{u}^{1+\beta}\,dx\big)(\tau)\,d\tau\,d\xi
\nonumber \\
\geq & \,(1+\beta)\int_{0}^{s}\varphi(\xi)\left[\int_{B_{r}} u_{0} \psi^2
\tilde{u}^{\beta} dx  k(\xi+t_0) + \int_{B_{r}}\int_{0}^{t_{0}}[-\dot{k}(\xi+t_0-\tau)]u(\tau,x)d\tau \tilde{u}^{\beta}\psi^{2} dx\right]d\xi.\nonumber\\
\end{align*}
We evaluate this inequality at $s=\sigma\eta\vdr$ and make use of the properties of $\varphi$ to find that
\[
\int_{0}^{\frac{\sigma+\sigma'}{2}\eta\vdr}\!\!k(\sigma\eta\vdr-\xi)\int_{B_{r}}\psi^2\tilde{u}^{1+\beta}(\xi,x)\,dx d\xi + \frac{c(\Lambda, \nu)(\beta+1)}{|\beta|} \frac{1}{r^{2}(\sigma-\sigma')^2}\,\int_{0}^{\frac{\sigma+\sigma'}{2}\eta\vdr}\int_{B_{r\sigma }}
\tilde{u}^{\beta+1}(\xi,x)\,dx ds
\]
\[
- \int_0^{\sigma\eta\vdr}\!\!\!\dot{\varphi}(\xi)\int_0^\xi
k(\xi-\tau)
\big(\int_{B_{r}}\psi^2\tilde{u}^{1+\beta}\,dx\big)(\tau)\,d\tau\,d\xi\nonumber \\
\]
\eqq{
 \geq (1+\beta)\int_{0}^{\sigma'\eta\vdr}\left[\int_{B_{r}} u_{0} \psi^2
\tilde{u}^{\beta} dx  k(\xi+t_0) + \int_{B_{r}}\int_{0}^{t_{0}}[-\dot{k}(\xi+t_0-\tau)]u(\tau,x)d\tau \tilde{u}^{\beta}\psi^{2} dx\right]d\xi.
}{infa}
We will estimate each term separately.
At first, since $k$ is nonincreasing, for any $0 <\xi < \sigma'\eta\vdr$ we have $k(\xi+t_0) \geq k(\eta\vdr+t_0)$, hence
\[
\int_{0}^{\sigma'\eta\vdr}\int_{B_{r}} u_{0} \psi^2
\tilde{u}^{\beta} dx\, k(\xi+t_0) d\xi \geq  k(\eta\vdr+t_0)\int_{0}^{\sigma'\eta\vdr}\int_{B_{r\sigma'}} u_{0}
\tilde{u}^{\beta}\, dx d\xi
\]
\eqq{
\geq   k(\eta\vdr+t_0)\int_{B_{r\sigma'}} u_{0}dx \inf_{x \in B_{r\sigma'}} \int_{0}^{\sigma'\eta\vdr}\tilde{u}^{\beta}(\xi,x)  d\xi.
}{infb}
Then, using convexity of $k$ we have
\[
\int_{0}^{\sigma'\eta\vdr}\int_{B_{r}}\int_{0}^{t_{0}}[-\dot{k}(\xi+t_0-\tau)]u(\tau,x) d\tau\, \tilde{u}^{\beta} \psi^{2}dxd\xi
\]
\eqq{
\geq  \inf_{x \in B_{r\sigma' }} \int_{0}^{\sigma'\eta\vdr}\tilde{u}^{\beta}(\xi,x) d\xi
\int_{0}^{t_{0}}[-\dot{k}(t_0+\sigma'\eta\vdr-\tau)]\int_{B_{r\sigma'}}u(\tau,x)dxd\tau.
}{infc}
Furthermore, using the properties of $\varphi$, monotonicity of $k$  and Lemma \ref{kernels}, we may estimate
\[
- \int_0^{\sigma\eta\vdr}\!\!\!\dot{\varphi}(\xi)\int_0^\xi
k(\xi-\tau)
\big(\int_{B_{r}}\psi^2\tilde{u}^{1+\beta}\,dx\big)(\tau)\,d\tau\,d\xi
\]
\eqq{
 \leq \frac{4(1*k)(\sigma\eta\vdr)}{(\sigma-\sigma')\eta\vdr} \int_{0}^{\sigma\eta\vdr}\int_{B_{r}}\psi^2\tilde{u}^{1+\beta}\,dxd\xi \leq \frac{4 \bar{c}\max\{1,\eta^{-1}\}}{\sigma-\sigma'}k_{1}(\vdr)\int_{0}^{\sigma\eta\vdr}\int_{B_{r}}\psi^2\tilde{u}^{1+\beta}\,dxd\xi.
}{infd}
Finally, by monotonicity of $k$, Lemma \ref{kernels} and Lemma \ref{ba}
\[
\int_{0}^{\frac{\sigma+\sigma'}{2}\eta\vdr}k(\sigma\eta\vdr-\xi)\int_{B_{r}}\psi^2\tilde{u}^{1+\beta}(s)\,dx d\xi \leq k\big(\frac{\sigma-\sigma'}{2}\eta\vdr\big)\int_{0}^{\frac{\sigma+\sigma'}{2}\eta\vdr}\int_{B_{r}}\psi^2\tilde{u}^{1+\beta}\,dx d\xi\]
\eqq{
\leq \frac{2\max\{1,\eta^{-1}\}}{\sigma'-\sigma}k_{1}(\vdr)\int_{0}^{\sigma\eta\vdr}\int_{B_{r}}\psi^2\tilde{u}^{1+\beta}\,dx d\xi.
}{infe}
Using (\ref{infb})-(\ref{infe}) and Lemma \ref{fi} in (\ref{infa}) we arrive at
\[
\left(\int_{0}^{t_{0}}[-\dot{k}(t_0+\sigma'\eta\vdr-\tau)]\int_{B_{r\sigma'}}u(\tau,x)dxd\tau + k(\eta\vdr+t_0) \int_{B_{r\sigma'}} u_{0}dx \right) \inf_{x \in B_{r\sigma' }} \int_{0}^{\sigma'\eta\vdr}\tilde{u}^{\beta}(\xi,x) d\xi
\]
\[
\leq \frac{1}{|\beta|(\beta+1)}\frac{C}{(\sigma-\sigma')^{2}}\frac{1}{r^{2}}\int_{0}^{\sigma\eta\vdr}\int_{B_{r\sigma}}\tilde{u}^{1+\beta}\,dxd\xi,
\]
where $C>0$ depends on $\overline{c},\eta,\Lambda, \nu$. Thus,
\[
\int_{0}^{t_{0}}[-\dot{k}(t_0+\sigma'\eta\vdr-\tau)]\int_{B_{\sigma'r}}u(\tau,x)dxd\tau + k(\eta\vdr+t_0) \int_{B_{\sigma'r}} u_{0}dx
\]
\[
\leq \frac{1}{|\beta|(\beta+1)}\frac{C}{(\sigma-\sigma')^{2}r^{2}} \sup_{x \in B_{r\sigma' }} \left(\int_{0}^{\sigma'\eta\vdr}\tilde{u}^{\beta}(\xi,x)d\xi\right)^{-1}\int_{0}^{\sigma\eta\vdr}\int_{B_{r\sigma}}\tilde{u}^{1+\beta}\,dxd\xi.
\]
Applying Jensen's inequality we get
\[
\int_{0}^{t_{0}}[-\dot{k}(t_0+\sigma'\eta\vdr-\tau)]\int_{B_{\sigma'r}}u(\tau,x)dxd\tau + k(\eta\vdr+t_0) \int_{B_{\sigma'r}} u_{0}dx
\]
\eqq{
 \leq \frac{1}{|\beta|(\beta+1)} \frac{C}{(\sigma-\sigma')^2r^{2}}\frac{1}{(\sigma'\eta\vdr)^{2}} \sup_{x \in B_{r\sigma' }}\int_{0}^{\sigma'\eta\vdr}\tilde{u}^{-\beta}(\xi,x)d\xi\int_{0}^{\sigma\eta\vdr}\int_{B_{r\sigma }}\tilde{u}^{1+\beta}\,dx d\xi.
}{uz1}
Since  $1+\beta \in (0,1)$, we may apply the weak Harnack estimate (\cite[Theorem 1.1]{krz}) to the result
\eqq{
\frac{1}{r\vdr}\int_{0}^{\sigma\eta\vdr}\int_{B_{r\sigma }}\tilde{u}^{1+\beta}(\xi,x)\,dx d\xi
\leq C \left(\inf_{((2-\sigma)\eta\vdr,2\eta\vdr)\times B_{r\sigma }} \tilde{u}\right)^{1+\beta},
}{uz2}
where $C=C(\nu,\Lambda,\sigma,\eta,\bar{c},\tilde{c},p_0)$. Let us denote $w = \tilde{u}^{-\frac{\beta}{2}}$. From the interpolation inequality in one dimension we infer
\[
\sup_{x \in B_{r\sigma' }}\int_{0}^{\sigma'\eta\vdr}\tilde{u}^{-\beta}(\xi,x)d\xi \leq \int_{0}^{\sigma'\eta\vdr} \norm{w(\xi,\cdot)}^{2}_{L_{\infty}(B_{r \sigma'})}d\xi
\]
\[
\leq c\int_{0}^{\sigma'\eta\vdr} \Big(\norm{Dw(\xi,\cdot)}^{\frac{4}{3}}_{L_{2}(B_{r\sigma' })}\norm{w(\xi,\cdot)}^{\frac{2}{3}}_{L_{1}(B_{r\sigma' })} + \norm{w(\xi,\cdot)}^{2}_{L_{2}(B_{r\sigma' })}\Big)\,d\xi.
\]
Applying H\"older's inequality with the pair of exponents $(3/2,3)$ we obtain further
\begin{align}
\sup_{x \in B_{r\sigma' }}\int_{0}^{\sigma'\eta\vdr}\tilde{u}^{-\beta}(\xi,x)d\xi & \le
 c \left(\int_{0}^{\sigma'\eta\vdr} \norm{Dw(\xi,\cdot)}^{2}_{L_{2}(B_{r\sigma' })}d\xi\right)^{\frac{2}{3}}\left(\int_{0}^{\sigma'\eta\vdr} \norm{w(\xi,\cdot)}^{2}_{L_{1}(B_{r\sigma' })}d\xi\right)^{\frac{1}{3}}
 \nonumber\\
 &\quad + c\int_{0}^{\sigma'\eta\vdr}\norm{w(\xi,\cdot)}_{L_{2}(B_{r\sigma'})}^{2}d\xi.
\label{uz4}
\end{align}
Note that for supersolutions of (\ref{MProb}), applying estimates  from the proof of \cite[Theorem 3.2]{krz} with appropriately chosen cylinders (see the estimate between (70) and (71) together with inequality (49), here $q_1=\infty, q_2 = \infty$) gives
\[
\norm{Dw}_{L_{2}(0,\sigma'\eta\vdr;L_{2}(B_{r\sigma' }))} \leq  \frac{C}{r(\sigma-\sigma')|\beta+1|} \norm{w}_{L_{2}(0,\sigma\eta\vdr;L_{2}(B_{r\sigma }))},
\]
where $C$ is a positive constant depending on $\Lambda,\nu,\eta, \delta, \overline{c}$. Since
\eqq{
\norm{w}_{L_{2}(0,\sigma\eta\vdr;L_{2}(B_{r\sigma }))}^{2} = \int_{0}^{\sigma\eta\vdr}\int_{B_{r\sigma }}\tilde{u}^{-\beta}(\xi,x)\,dx d\xi,
}{uz98}
we have
\eqq{
\left(\int_{0}^{\sigma'\eta\vdr} \norm{Dw(\xi,\cdot)}^{2}_{L_{2}(B_{r\sigma' })}d\xi\right)^{\frac{2}{3}} \leq \frac{C}{[|\beta+1|(\sigma-\sigma')r]^{\frac{4}{3}}} \left( \int_{0}^{\sigma\eta\vdr}\int_{B_{r\sigma }}\tilde{u}^{-\beta}(\xi,x)\,dx d\xi\right)^{\frac{2}{3}},
}{uz6}
where the constants $C$ depends only on $\Lambda,\nu,\eta, \delta, \overline{c}$. Next, H\"older's inequality implies
\eqq{
\left(\int_{0}^{\sigma'\eta\vdr} \norm{w(\xi,\cdot)}^{2}_{L_{1}(B_{r\sigma' })}d\xi\right)^{\frac{1}{3}}  \leq  r^{\frac{1}{3}} \left(\int_{0}^{\sigma'\eta\vdr}\int_{B_{r\sigma' }}\tilde{u}^{-\beta}(\xi,x)\,dx d\xi\right)^{\frac{1}{3}}.
}{uz7}
Using (\ref{uz98}), (\ref{uz6}) and (\ref{uz7}) in (\ref{uz4})
leads to
\[
\sup_{x \in B_{r\sigma' }}\int_{0}^{\sigma'\eta\vdr}\tilde{u}^{-\beta}(\xi,x)d\xi \leq \frac{C}{[|\beta+1|(\sigma-\sigma')]^{\frac{4}{3}}} \frac{1}{r}\int_{0}^{\sigma\eta\vdr}\int_{B_{r\sigma }}\tilde{u}^{-\beta}(\xi,x)\,dx d\xi.
\]
Recalling that $-\beta \in (0,1)$ we apply the weak Harnack inequality (\cite[Theorem 1.1]{krz}) to the result
\eqq{
\sup_{x \in B_{r\sigma' }}\frac{1}{\vdr}\int_{0}^{\sigma'\eta\vdr}\tilde{u}^{-\beta}(s,x)ds \leq \frac{C}{[|\beta+1|(\sigma-\sigma')]^{\frac{4}{3}}} \left(\inf_{((2-\sigma)\eta\vdr,2\eta\vdr)\times B_{r\sigma }} \tilde{u}\right)^{-\beta},
}{uz8}
where $C=C(\Lambda,\nu,\eta, \sigma,\delta, \overline{c},\tilde{c},p_0)$. We fix $\beta \in (-1,0)$. Then (\ref{uz2}) together with (\ref{uz8}) applied in (\ref{uz1}) leads to
\[
\int_{0}^{t_{0}}[-\dot{k}(t_0+\sigma'\eta\vdr-\tau)]\int_{B_{\sigma'r}}u(\tau,x)dxd\tau + k(\eta\vdr+t_0) \int_{B_{\sigma'r}} u_{0}dx
\]
\[
\leq \frac{C}{r} (\sigma-\sigma')^{-4}  \inf_{((2-\sigma)\eta\vdr,2\eta\vdr)\times B_{r\sigma }} \tilde{u},
\]
which finishes the proof.
\end{proof}
\begin{bemerk1} \label{RemA}
{\em Note that up to estimate \eqref{uz2}, the above line of arguments is valid for all space dimensions (up to some minor adjustments). The assumption on the space dimension plays a crucial role when applying the interpolation inequality. The embedding $H^1_2\hookrightarrow L_\infty$ (on balls), which is used here, is only true in dimension one.
}
\end{bemerk1}
\section{Estimates for subsolutions} \label{mvi}

\begin{satz}\label{subest2}
Let $T>0$, and $\Omega\subset \iR$ be a bounded
domain. Suppose the assumptions (H1)--(H3) and (\ref{pc}), (\ref{ak1}), (\ref{asholder}) are satisfied. Let
$\eta>0$ and $\delta\in (0,1)$ be fixed. Then there exists $r^{*} \in (0,r_0)$ such that for any $0<r \leq r^{*}$, any $t_0\in(0,T]$
 with $t_0-\eta \vdr \ge 0$, any ball
$B_r(x_{0})\subset\Omega$, and any weak subsolution $u\ge 0$ of (\ref{MProb}) in $(0,t_0)\times B_r(x_{0})$ with $u_0\ge 0$
in $B_r(x_{0})$, there holds
\[
\frac{1}{\vdr}\int_{t_0-\sigma'\eta\vdr}^{t_0}  \sup_{x\in B_{r\sigma' }} u(t,x) \, dt \leq  \frac{C}{\sigma-\sigma'}\left(\frac{1}{r \vdr}\int_{t_0-\sigma\eta\vdr}^{t_{0}}\int_{B_{ r\sigma}}u^{2}\,dxdt\right)^{\frac{1}{2}}
\]
\[
+ C r  \left( k(t_0-\sigma\eta\vdr)\int_{B_{r\sigma}}u_{0}\,dx + \frac{1}{(\sigma-\sigma')^{\omega}}\int_{0}^{t_0-\sigma'\eta\vdr}[-\dot{k}(t_0-\tau)]\int_{B_{r\sigma}}u(\tau,x)\,dxd\tau\right).
\]

Here
$0<\delta\leq \sigma'<\sigma\le 1$, $C=C(\nu,\Lambda,\delta,\eta,\overline{c})$ and $\omega$ comes from (\ref{asholder}).
\end{satz}

\begin{proof}
Again, to abbreviate the notation we denote $B_{r} := B_{r}(x_{0})$.
We consider $r \in (0,r^{*})$, where $r^{*}>0$ comes from Lemma \ref{scaling}. We fix $\sigma'$ and $\sigma$ such that $\delta\le \sigma'<\sigma\le
1$, and for any $\frac{\sigma'}{\sigma}<\rho'<\rho\le 1$, we introduce
$t_{1} = t_{0}-\rho\sigma\eta\vdr$ and $t_{2} = t_{0}-\rho'\sigma\eta\vdr$. Then $0\le t_1<t_2<t_0$. We
shift the time by setting ${s}=t-t_1$ and  $\tilde{f}(s)=f(s+t_1)$, $s\in
(0,t_0-t_1)$, for functions $f$ defined on $(t_1,t_0)$. Since
 $u$ is a weak subsolution of
(\ref{MProb}) in $(0,t_0)\times B_r$, due to (\ref{shiftprob}) for a.a. $s\in
(0,t_0-t_1)$ we have
\[
\int_{B_r} \Big(v \partial_s(k_{n}\ast \tilde{u})+\big(h_n\ast
[\tilde{A}D\tilde{u}]|Dv\big)\Big)\,dx
\]
\begin{equation} \label{sup0}
\leq   \int_{B_r} u_{0} v\, dx\, (h_n*k(\cdot+t_1))  + \int_{B_r}h_n*\left(\int_{0}^{t_{1}}[-\dot{k}(\cdot+t_1-\tau)]u(\tau,x)\,d\tau \right) v\, dx ,
\end{equation}
for any nonnegative function $v\in \oH^1_2(B_r)$. In (\ref{sup0}) we choose the test function $v=\psi^2
\tilde{u}$, where
 $\psi\in C^1_0(B_{\rho r})$ so that
$0\le \psi\le 1$, $\psi=1$ in $B_{\rho'r\sigma }$, supp$\,\psi\subset B_{\rho r\sigma }$, and $|D \psi|\le 2/[r\sigma(\rho-\rho')]$.
The fundamental
identity (\ref{fundidentity}), applied to $k=k_{n}$ and the
convex function $H(y)=\frac{1}{2}y^{2}$, $y>0$, implies that for a.a. $(s,x)\in (0,t_0-t_1)\times B_r$
\begin{align}
 \tilde{u}\partial_{s}(k_{n}\ast \tilde{u}) & \ge \,\frac{1}{2}\,\partial_{s}
(k_{n}\ast\tilde{u}^{2}). \label{sup1}
\end{align}
Applying
\[ Dv=2\psi D\psi \,\tilde{u}+\psi^2 D \tilde{u}\]
 together with (\ref{sup1}) in (\ref{sup0}), we obtain for
a.a. $s\in (0,t_0-t_1)$
\[
\frac{1}{2} \int_{B_{r\sigma }}\psi^2\partial_{s}
(k_{n}\ast\tilde{u}^{2})\,dx+\int_{B_{r\sigma }}\big(h_n\ast
[\tilde{A}D\tilde{u}]|\psi^2 D
\tilde{u}\big)\,dx \]
\[
\le -2\int_{B_{r\sigma }}\big(h_n\ast [\tilde{A}D\tilde{u}]|\psi D\psi
\,\tilde{u}\big)\,dx
\]
\eqq{
+\int_{B_{r\sigma }}\psi^2\tilde{u}u_{0}\,dx\, (h_n*k(\cdot+t_1))+ \int_{B_{r\sigma }} h_n*\left( \int_{0}^{t_{1}}[-\dot{k}(\cdot+t_1-\tau)]u(\tau,x)d\tau \right)\psi^2\tilde{u}\,dx.
}{sup2}

Now, we choose another cut-off function $\varphi\in C^1([0,t_0-t_1])$ with the following properties: $0\le
\varphi\le 1$, $\varphi=0$ in $[0,(t_2-t_1)/2]$, $\varphi=1$ in
$[t_2-t_1,t_0-t_1]$, and $0\le \dot{\varphi}\le 4/(t_2-t_1)$.
Multiplying (\ref{sup2}) by $2\varphi(s)$,
and convolving the resulting inequality with $l$ yields
\begin{align}
\int_{B_{r\sigma }} & l\ast
\big(\varphi\partial_{s}(k_{n}\ast
[\psi^2\tilde{u}^{2}])\big)\,dx+2\,l\ast\int_{B_{r\sigma }}\big(h_n\ast
[\tilde{A}D\tilde{u}]|\psi^2 D
\tilde{u}\big)\varphi\,dx \nonumber\\
\le \, & \,-4\,l\ast\int_{B_{r\sigma }}\big(h_n\ast
[\tilde{A}D\tilde{u}]|\psi D\psi
\,\tilde{u}\big)\varphi\,dx+2\,l\ast\left(\int_{B_{r\sigma }}\psi^2\tilde{u}u_{0}
\,dx (h_n*k(\cdot+t_1))\varphi\right)\nonumber\\
&  +2 \int_{B_{r \sigma }}l*\left(h_n*\left(\int_{0}^{t_{1}}[-\dot{k}(\cdot+t_1-\tau)]u(\tau,x)d\tau \right)\varphi \tilde{u} \psi^2\right)dx, \label{sup3}
\end{align}
 a.e. in $(0,t_0-t_1)$. Lemma \ref{comm} implies
\begin{align}
\int_{B_{r\sigma}} l\ast &
\big(\varphi\partial_{s}(k_{n}\ast
[\psi^2\tilde{u}^{2}])\big)\,dx \ge \int_{B_{r\sigma}} \varphi
l\ast \big(\partial_{s}(k_{n}\ast
[\psi^2\tilde{u}^{2}])\big)\,dx\nonumber\\
& -\int_0^s l(s-\xi)\dot{\varphi}(\xi)
\big(k_{n}\ast
\int_{B_{r\sigma}}\psi^2\tilde{u}^{2}\,dx\big)(\xi)\,d\xi.
\label{sup4}
\end{align}
Furthermore, from $l*k=1$, $k_{n}=k\ast h_n$  and
\[
k_{n}\ast [\psi^2\tilde{u}^{2}]\in
{}_0H^1_1([0,t_0-t_1];L_1(B_{r\sigma}))
\]
 we have
\begin{equation} \label{sup5}
l\ast \partial_{s}(k_{n}\ast
[\psi^2\tilde{u}^{2}])=\partial_s(l\ast
k_{n}\ast [\psi^2\tilde{u}^{2}])=h_n\ast
(\psi^2\tilde{u}^{2}).
\end{equation}
If we combine (\ref{sup3}), (\ref{sup4}) and (\ref{sup5}), and pass to the limit with
$n$ (choosing an appropriate subsequence, if
necessary), we arrive at
\begin{align}
& \int_{B_{r\sigma}}\varphi\psi^2\tilde{u}^{2}\,dx+
2\,l\ast\int_{B_{r\sigma}}\big(\tilde{A}D\tilde{u}|\psi^2
D\tilde{u}\big)\varphi\,dx\nonumber\\
\le \, &
\,-4\,l\ast\int_{B_{r\sigma}}\big(\tilde{A}D\tilde{u}|\psi
D\psi
\,\tilde{u}\big)\varphi\,dx+2\,l\ast(\int_{B_{r\sigma}}\psi^2\tilde{u}u_{0}\, dx\,
k(\cdot+t_1)\varphi) \nonumber\\
& +l*(\dot{\varphi}
k\ast
\int_{B_{r\sigma}}\psi^2\tilde{u}^{2}\,dx)  \nonumber\\
& +2 \int_{B_{r\sigma}}l*\left(\int_{0}^{t_{1}}[-\dot{k}(\cdot+t_1-\tau)]u(\tau,x)\,d\tau\, \varphi \tilde{u} \psi^2\right)dx,
\quad\mbox{a.a.}\;s\in(0,t_0-t_1).
\label{sup6}
\end{align}

Making use of (H1) and (H2) we find that
\eqq{
\int_{B_{r\sigma}}\varphi\psi^2\tilde{u}^2\,dx+\,\nu \,l\ast\int_{B_{r\sigma}}\varphi \psi^2|D\tilde{u}|^2\,dx
\le l\ast F,\quad\mbox{a.a.}\;s\in(0,t_0-t_1),
}{uz11}
where
\eqnsl{
F(s) =&\,  \,\frac{4\Lambda^2}{\nu}\, \int_{B_{r\sigma}}
|D\psi|^2\varphi \tilde{u}^2\,dx
+\dot{\varphi}(s)
\big(k\ast
\int_{B_{r\sigma}}\psi^2 \tilde{u}^{2}\,dx\big)(s) \\
&
+ 2\,\int_{B_{r\sigma}}\psi^2 \tilde{u} u_{0}\, dx\,
k(s+t_1)\varphi(s) + 2\int_{B_{r\sigma}}\varphi \tilde{u} \psi^2\int_{0}^{t_{1}}[-\dot{k}(s+t_1-\tau)]u(\tau,x)\,d\tau dx.
}{defeF}
Dropping the second term on the left-hand side of \eqref{uz11} and applying Young's inequality for convolutions and the properties of
$\varphi$ we then infer that
\eqq{
\int_{t_2-t_1}^{t_0-t_1} \norm{
\tilde{u}\psi}^2_{L_{2}(B_{r\sigma})}\,ds \,\le
\|l\|_{L_1([0,t_0-t_1])} \int_0^{t_0-t_1} \!\!\!\!F(s)\,ds.
}{uz17}
On the other hand, dropping the first term on the left-hand-side of \eqref{uz11}, convolving
the resulting inequality with $k$ and evaluating at
$s=t_0-t_1$, we get
\eqq{
\int_{t_2-t_1}^{t_0-t_1}\int_{B_{r\sigma}}\psi^2|D\tilde{u}|^2\,dx\,ds \le
\,\frac{1}{\nu}\,\int_0^{t_0-t_1} \!\!\!\!F(s)\,ds.
}{uz18}
In order to estimate the $L_1-$ norm of $F$, we note that
\[
\dot{\vf}(s)(k*\int_{B_{r\sigma}}\psi^{2}\tilde{u}^{2}dx)(s)\leq \frac{4}{\sigma\eta(\rho-\rho')\vdr}(k*\int_{B_{\rho r\sigma }}\tilde{u}^{2}dx)(s)
\]
and consequently, by Young's inequality for convolutions we get
\[
\int_{0}^{t_{0}-t_{1}}\dot{\vf}\cdot k*\int_{B_{\rho r\sigma }}\psi^{2}\tilde{u}^{2}\,dxds \leq  \frac{4}{\sigma\eta(\rho-\rho')\vdr} (1*k*\int_{B_{\rho r\sigma }}\tilde{u}^{2}dx)(t_0-t_1)
\]
\[
 \leq \frac{4}{\sigma\eta(\rho-\rho')\vdr} (1*k)(\rho\sigma\eta\vdr)\int_{0}^{t_{0}-t_{1}}\int_{B_{\rho r\sigma }}\tilde{u}^{2}\,dxds
\]
\eqq{
\leq \frac{4}{(\rho-\rho')}  \frac{(1*k)(\sigma\eta \vdr)}{\sigma\eta \vdr}\int_{0}^{t_{0}-t_{1}}\int_{B_{\rho r\sigma }}\tilde{u}^{2}dxds
\leq  \frac{4 \overline{c} \max\{1,\eta^{-1}\}}{\sigma(\rho-\rho')} \ki(\vdr)
 \int_{0}^{t_{0}-t_{1}}\int_{B_{\rho r\sigma }}\tilde{u}^{2}dxds,
}{uz13}
where in the last estimate we used the fact that $k$ is nonincreasing and Lemma \ref{kernels}.
Then
\eqq{
\int_{0}^{t_{0}-t_{1}}\int_{B_{r\sigma}}\abs{D \psi}^{2}\tilde{u}^{2}\,dxds \leq \frac{4}{r^{2}\sigma^{2}(\rho-\rho')^{2}}\int_{0}^{t_{0}-t_{1}}\int_{B_{\rho r\sigma}}\tilde{u}^{2}\,dxds.
}{pierw1}
Further, we have
\eqq{
\int_{0}^{t_0-t_1} \int_{B_{r\sigma}}\psi^2 \tilde{u} u_{0} \,dx\, k(s+t_1)\varphi(s)\,ds \leq k(t_1)\int_{B_{r\sigma}}u_{0}\,dx \int_{\frac{t_2-t_1}{2}}^{t_0-t_1}  \norm{\tilde{u}\psi}_{L_{\infty}(B_{r\sigma})}ds.
}{uz14}
We now come to the last term of $F$.
Using firstly the convexity of $k$ and then the assumption (\ref{asholder}) for $s\in (\frac{t_2-t_1}{2},t_0-t_1)$ and $\tau \in [0,t_0-\sigma'\eta\vdr]$ we may estimate
\[
-\dot{k}(s+t_1-\tau)  \leq -\dot{k}\big(t_0-\tau - \sigma\eta\vdr\frac{\rho+\rho'}{2}\big)\leq -\dot{k}\left((t_0-\tau)(1-\frac{\sigma\eta\vdr(\rho+\rho')}{2(t_0-\tau)})\right)
\]
\[
 \leq c(\eta) [-\dot{k}(t_0-\tau)] \left(1-\frac{\rho+\rho'}{2\rho})\right)^{-\omega} \leq \frac{c(\eta)}{(\rho'-\rho)^{\omega}}[-\dot{k}(t_0-\tau)].
\]

Hence, using again $\rho \geq \frac{\sigma'}{\sigma}$ we see that
\[
 \int_{0}^{t_0-t_1}\int_{B_{r\sigma}}\varphi(s) \tilde{u}(s,x) \psi(x)^2\int_{0}^{t_{1}}[-\dot{k}(s+t_1-\tau)]u(\tau,x)\,d\tau dxds
\]
\eqq{
\leq \frac{c(\eta)}{(\rho'-\rho)^{\omega}} \int_{\frac{t_2-t_1}{2}}^{t_0-t_1}  \norm{\tilde{u}\psi}_{L_{\infty}(B_{r\sigma})}ds\int_{0}^{t_0-\sigma'\eta\vdr}[-\dot{k}(t_0-\tau)]\int_{B_{r\sigma}}\psi(x) u(\tau,x)\,dxd\tau.
}{uz15}
Combining (\ref{uz13}) - (\ref{uz14}) we arrive at
\[
\int_{0}^{t_0-t_1}F(s)ds \leq    \frac{C}{r^2\sigma^{2}(\rho-\rho')^{2}}\int_{0}^{t_{0}-t_{1}}\int_{B_{ \rho r\sigma}}\tilde{u}^{2}\,dxds
\]
\eqq{
+ C \int_{\frac{t_2-t_1}{2}}^{t_0-t_1}  \norm{\tilde{u}\psi}_{L_{\infty}(B_{r\sigma})}ds
\cdot  \left( k(t_1)\int_{B_{r\sigma}}u_{0}\,dx + \frac{1}{(\rho-\rho')^{\omega}}\int_{0}^{t_0-\sigma'\eta\vdr}[-\dot{k}(t_0-\tau)]\int_{B_{r\sigma}}u(\tau,x)dxd\tau\right)
}{uz20}
with positive $C$ depending on $\nu,\Lambda, \eta,\delta$. Applying interpolation inequality and then using estimates (\ref{uz17}), (\ref{uz18}), (\ref{pierw1}), (\ref{uz20}) and  Lemma \ref{scaling} we find that
\[
\int_{t_2-t_{1}}^{t_0-t_1}  \norm{\tilde{u}\psi}_{L_{\infty}(B_{r\sigma})}^{2}ds \leq c\norm{D(\tilde{u}\psi)}_{L_{2}((t_2-t_{1},t_0-t_1);B_{r\sigma})}
\norm{\tilde{u}\psi}_{L_{2}((t_2-t_{1},t_0-t_1);B_{r\sigma})}
\]
\[
\leq  \frac{C}{r\sigma^{2}(\rho-\rho')^{2}}\int_{0}^{t_{0}-t_{1}}\int_{B_{\rho r\sigma}}\tilde{u}^{2}dxds + r \int_{\frac{t_2-t_1}{2}}^{t_0-t_1}  \norm{\tilde{u}\psi}_{L_{\infty}(B_{r\sigma})}ds
\]
\[
\times C \left( k(t_1)\int_{B_{r\sigma}}u_{0}dx + \frac{1}{(\rho-\rho')^{\omega}}\int_{0}^{t_0-\sigma'\eta\vdr}[-\dot{k}(t_0-\tau)]\int_{B_{r\sigma}}u(\tau,x)\,dxd\tau\right),
\]
where $C=C(\nu,\Lambda,\overline{c},\eta,\delta)$.
Applying Jensen's inequality we then deduce further that
\[
\frac{1}{\eta^{\frac{1}{2}}\vdr^{\frac{1}{2}}}\int_{t_2-t_{1}}^{t_0-t_1}  \norm{\tilde{u}}_{L_{\infty}(B_{r\sigma\rho'})}ds \leq \frac{C}{\sigma(\rho-\rho')}\left(\frac{1}{r}\int_{0}^{t_{0}-t_{1}}\int_{B_{\rho r\sigma}}\tilde{u}^{2}dxds\right)^{\frac{1}{2}}
\]
\[
+ C\left(\int_{\frac{t_2-t_1}{2}}^{t_0-t_1}  \norm{\tilde{u}}_{L_{\infty}(B_{r\sigma\rho})}ds\right)^{\frac{1}{2}} r^{\frac{1}{2}}\left( k(t_1)\int_{B_{r\sigma}}\!\!u_{0}dx +\frac{1}{(\rho-\rho')^{\omega}} \int_{0}^{t_0-\sigma'\eta\vdr}\!\!\!\!\!\![-\dot{k}(t_0-\tau)]\int_{B_{r\sigma}}\!\!\!u(\tau,x)dxd\tau\right)^{\frac{1}{2}}.
\]
Using Young's inequality
and coming back to the original variable we get
\[
\int_{t_2}^{t_0}  \norm{u}_{L_{\infty}(B_{r\sigma\rho'})}dt \leq \frac{1}{2}\int_{\frac{t_2+t_1}{2}}^{t_0}  \norm{u}_{L_{\infty}(B_{r\sigma\rho})}dt + \frac{C \vdr^{\frac{1}{2}}}{\sigma(\rho-\rho')}\left(\frac{1}{r}\int_{t_1}^{t_{0}}\int_{B_{ r\sigma}}u^{2}dxdt\right)^{\frac{1}{2}}
\]
\[
+C r \vdr \left( k(t_1)\int_{B_{r\sigma}}u_{0}dx + \frac{1}{(\rho-\rho')^{\omega}}\int_{0}^{t_0-\sigma'\eta\vdr}[-\dot{k}(t_0-\tau)]\int_{B_{r\sigma}}u(\tau,x)dxd\tau\right).
\]
Using $t_2+t_1 \geq 2t_1$, $\rho \leq 1$ and denoting
\[
f(\rho) = \int_{t_0-\rho\sigma\eta\vdr}^{t_0}  \norm{u}_{L_{\infty}(B_{r\sigma\rho})}dt
\]
we obtain
\[
f(\rho') \leq \frac{1}{2}f(\rho) +  \frac{C \vdr^{\frac{1}{2}}}{\sigma(\rho-\rho')}\left(\frac{1}{r}\int_{t_0-\sigma\eta\vdr}^{t_{0}}\int_{B_{ r\sigma}}u^{2}dxdt\right)^{\frac{1}{2}}
\]
\[
+ C r \vdr \left( k(t_0-\sigma\eta\vdr)\int_{B_{r\sigma}}u_{0}dx + \frac{1}{(\rho-\rho')^{\omega}}\int_{0}^{t_0-\sigma'\eta\vdr}[-\dot{k}(t_0-\tau)]\int_{B_r\sigma}u(\tau,x)dxd\tau\right).
\]
Applying Lemma \ref{geometric}, we arrive at
\[
\int_{t_0-\sigma'\eta\vdr}^{t_0}  \norm{u}_{L_{\infty}(B_{r\sigma' })}dt \leq C \frac{\vdr^{\frac{1}{2}}}{\sigma-\sigma'}\left(\frac{1}{r}\int_{t_0-\sigma\eta\vdr}^{t_{0}}\int_{B_{ r\sigma}}u^{2}dxdt\right)^{\frac{1}{2}}
\]
\[
+ C r \vdr \left( k(t_0-\sigma\eta\vdr)\int_{B_{r\sigma}}u_{0}dx + \frac{1}{(\sigma-\sigma')^{\omega}}\int_{0}^{t_0-\sigma'\eta\vdr}[-\dot{k}(t_0-\tau)]\int_{B_{r\sigma}}u(\tau,x)dxd\tau\right),
\]
where $C=C(\nu, \Lambda,\overline{c},\eta,\delta)$, and the proof is finished.
\end{proof}

\begin{satz} \label{subest1}
Let $T>0$, and $\Omega\subset \iR$ be a bounded
domain. Suppose the assumptions (H1)--(H3) and (\ref{pc}), (\ref{ak1}) and (\ref{asholder}) are satisfied. Let
$\eta>0$ and $\delta\in (0,1)$ be fixed. Then there exists $r^{*}>0$ such that for any $0<r \leq r^{*}$, any $t_0\in(0,T]$
 with $t_0-\eta \vdr \ge 0$, any ball
$B_r(x_{0})\subset\Omega$, and any weak subsolution $u\ge 0$ of (\ref{MProb}) in $(0,t_0)\times B_r(x_{0})$ with $u_0\ge 0$
in $B_r(x_{0})$, there holds
\begin{align*}
\esup_{U_{\sigma'}}{u} & \le \frac{C}{(\sigma-\sigma')^{\tau_0}} \Big(\frac{1}{\vdr}\int_{t_0-\sigma\eta\vdr}^{t_0} \sup_{x\in B_{r \sigma }} u  + k(t_0-\sigma\eta\vdr)\,r \int_{B_{r \sigma}} u_0\, dx \\
& \quad\quad
+r\int_{0}^{t_0-\sigma'\eta\vdr}[-\dot{k}(t_0-\tau)]\int_{B_{r\sigma}}u(\tau,x)\,dx\,d\tau\Big).
\end{align*}
Here $U_\sigma=(t_0-\sigma\eta \vdr,t_0)\times B_{r\sigma }(x_0)$,
$0<\delta\leq \sigma'<\sigma\le 1$, $C=C(\nu,\Lambda,\delta,\eta,\overline{c})$ and
$\tau_0 = \tau_0(\omega) > 0$.
\end{satz}
\noindent {\em Proof.}
We perturb the function $u$ by a memory term as follows. For $r \in (0,r_0)$ we set $u_b = u + b$ where
\[
b = k(t_0-\sigma\eta\vdr) r\int_{B_{r\sigma}} u_0 dx + r\int_{0}^{t_0-\sigma'\eta\vdr}[-\dot{k}(t_0-\tau)]\int_{B_{r\sigma}}u(\tau,x)\,dxd\tau.
\]
We proceed exactly as in the proof of the previous theorem and with the same notation. For any $\xi \in (0,1]$, we denote $V_\xi = U_{\xi\sigma}$.
\nic{
Again, to abbreviate the notation we denote $B_{r} := B_{r}(x_{0})$.
Let us discuss $r \in (0,r^{*})$, where $r^{*}>0$ comes from Lemma \ref{scaling}. Fix $\sigma'$ and $\sigma$ such that $\delta\le \sigma'<\sigma\le
1$, for $\rho\in (0,1]$ we set
$V_\rho=U_{\rho\sigma}$. For any $\frac{\sigma'}{\sigma}<\rho'<\rho\le 1$, we introduce
$t_{1} = t_{0}-\rho\sigma\eta\vdr$ and $t_{2} = t_{0}-\rho'\sigma\eta\vdr$. Then $0\le t_1<t_2<t_0$. We
shift the time by setting ${s}=t-t_1$ and  $\tilde{f}(s)=f(s+t_1)$, $s\in
(0,t_0-t_1)$, for functions $f$ defined on $(t_1,t_0)$.
}The shifted perturbation of $u$ satisfies
 for a.a. $s\in
(0,t_0-t_1)$
\[
\int_{B_r} \Big(v \partial_s(k_{n}\ast \tilde{u}_{b})+\big(h_n\ast
[\tilde{A}D\tilde{u}_b]|Dv\big)\Big)\,dx
\]
\begin{equation} \label{sup0a}
\leq   \int_{B_r} u_{0} v \,dx\, (h_n*k(\cdot+t_1))+ b\,k_{n}(s)\int_{B_{r}}vdx  + \int_{B_r}h_n*\left(\int_{0}^{t_{1}}[-\dot{k}(\cdot+t_1-\tau)]u(\tau,x)\,d\tau\right) v \,dx ,
\end{equation}
for any nonnegative function $v\in \oH^1_2(B_r)$.

We would like to test the equation with powers of the form $\tilde{u}_b^{\beta}$ with $\beta \geq 1$. In order to ensure that the functions to test with are admissible we introduce the truncation  $\tilde{u}_{b,m} = \min\{\tilde{u}_b,b+m\}$. For $m \geq 1$ we define the nonnegative function
\begin{align*}
H_m(y) =\left\{ \begin{array}{l@{\;:\;}l}
 \frac{y^{\beta+1}}{\beta+1} & 0\leq  y \leq m+b \\
\frac{(m+b)^{\beta+1}}{\beta+1} + \frac{(m+b)^{\beta-1}}{2}\,(y^2-(m+b)^2) & y > m+b.
\end{array} \right.
\end{align*}
$H_m \in C^{1}(\mathbb{R}_{+})$ and $H'_m(y)=y^{\beta}$ if $y \leq m+b$, and $H'_m(y)=(m+b)^{\beta-1}y$ for $y > m+b$. Thus $H_m$ is convex. Furthermore,  we have
\eqq{
H'_m(\tilde{u}_b) =\tilde{u}_{b,m}^{\beta-1}\tilde{u}_b, \hd \hd DH'_m(\tilde{u}_b)= \tilde{u}_{b,m}^{\beta-1}[(\beta-1)D\tilde{u}_{b,m} + D\tilde{u}_{b} ],
}{new2}
and the fundamental identity implies that for a.a. $(s,x)\in (0,t_0-t_1)\times B_r$
\begin{align}
H_m'(\tilde{u}_b)\partial_{s}(k_{n}\ast \tilde{u}_{b}) & \ge \partial_{s}
(k_{n}\ast H_m(\tilde{u}_{b})). \label{supnew1}
\end{align}
Let $\psi$ be the function from the proof of Theorem \ref{subest2}. Now we may choose $v=\psi^2H'_m(\tilde{u}_b)$ in (\ref{sup0a}) and
using (\ref{new2}) and  (\ref{supnew1}) in (\ref{sup0a}) we obtain for
a.a. $s\in (0,t_0-t_1)$
\[
\int_{B_{r\sigma}}\psi^2\partial_{s}
\big(k_{n}\ast H_m(\tilde{u}_{b})\big)\,dx+\int_{B_{r\sigma}}\big(h_n\ast
[\tilde{A}D\tilde{u}_b]|\psi^2 \tilde{u}_{b,m}^{\beta-1}[(\beta-1)D\tilde{u}_{b,m} + D\tilde{u}_{b}]
\big)\,dx \]
\[
\le -2\int_{B_{r\sigma}}\big(h_n\ast [\tilde{A}D\tilde{u}_b]|\psi D\psi
\,\tilde{u}_{b,m}^{\beta-1}\tilde{u}_b\big)\,dx+ k_{n} b\int_{B_{r\sigma}} \psi^{2}\tilde{u}_{b,m}^{\beta-1}\tilde{u}_{b} dx
\]
\eqq{
+\int_{B_{r\sigma}}\psi^2\tilde{u}_{b,m}^{\beta-1}\tilde{u}_{b}u_{0}\,dx\,(h_n*k(\cdot+t_1))+ \int_{B_{r\sigma}}\tilde{u}_{b,m}^{\beta-1}\tilde{u}_{b} \psi^2 h_n*\left(\int_{0}^{t_{1}}[-\dot{k}(s+t_1-\tau)]u(\tau,x)d\tau\right) dx.
}{supnew2}
Now, let $\varphi\in C^1([0,t_0-t_1])$ be as in the proof of Theorem \ref{subest2}.
Multiplying (\ref{supnew2}) by $\varphi(s)$
and convolving the resulting inequality with $l$ yields
\begin{align}
\int_{B_{r\sigma}} & l\ast
\big(\varphi\partial_{s}(k_{n}\ast
[\psi^2H_m(\tilde{u}_{b})])\big)\,dx+\,l\ast\int_{B_{r\sigma}}\big(h_n\ast
[\tilde{A}D\tilde{u}_b]|\psi^2 \tilde{u}_{b,m}^{\beta-1}[(\beta-1)D\tilde{u}_{b,m} + D\tilde{u}_{b}]\big)\varphi\,dx \nonumber\\
\le \, & \,-2\,l\ast\int_{B_{r\sigma}}\big(h_n\ast
[\tilde{A}D\tilde{u}_b]|\psi D\psi
\,\tilde{u}_{b,m}^{\beta-1}\tilde{u}_{b} \big)\varphi\,dx+\,l\ast \big(\int_{B_{r\sigma}}\psi^2\tilde{u}_{b,m}^{\beta-1}\tilde{u}_{b} u_{0}
\,dx\, (h_n*k(\cdot+t_1))\varphi\big)\nonumber\\
& +b\, l*\big(\varphi k_{n}\int_{B_{r\sigma}} \psi^{2}\tilde{u}_{b,m}^{\beta-1}\tilde{u}_{b}  dx\big) + l*\int_{B_{r\sigma}}\varphi \tilde{u}_{b,m}^{\beta-1}\tilde{u}_{b} \psi^2h_n*\left(\int_{0}^{t_{1}}[-\dot{k}(\cdot+t_1-\tau)]u(\tau,x)d\tau\right)dx, \label{supnew3}
\end{align}
for a.a. $s\in (0,t_0-t_1)$.

Applying the same reasoning as the one carried in (\ref{sup4}) - (\ref{sup6}), passing to the limit with
$n$ (with an appropriate subsequence, if
necessary), we arrive at
\begin{align}
& \int_{B_{r\sigma}}\varphi\psi^2H_m(\tilde{u}_{b})\,dx+
\,l\ast\int_{B_{r\sigma}}\big(\tilde{A}D\tilde{u}_{b}|\psi^2
\tilde{u}_{b,m}^{\beta-1}[(\beta-1)D\tilde{u}_{b,m} + D\tilde{u}_{b}]\big)\varphi\,dx\nonumber\\
\le \, &
\,-2\,l\ast\int_{B_{r\sigma}}\big(\tilde{A}D\tilde{u}_{b}|\psi
D\psi
\,\tilde{u}_{b,m}^{\beta-1}\tilde{u}_{b}\big)\varphi\,dx+\,l\ast\big(\int_{B_{r\sigma}}\psi^2\tilde{u}_{b,m}^{\beta-1}\tilde{u}_{b}u_{0}\, dx\,
k(\cdot+t_1)\varphi\big) \nonumber\\
& +\int_0^s l(s-\tau)\dot{\varphi}(\tau)
\big(k\ast
\int_{B_{r\sigma}}\psi^2H_m(\tilde{u}_{b})\,dx\big)(\tau)\,d\tau +b \int_{0}^{s}l(s-p)k(p)\vf(p)\int_{B_{r \sigma}} \psi^{2}\tilde{u}_{b,m}^{\beta-1}\tilde{u}_{b}dx dp \nonumber\\
& + l*\int_{B_{r\sigma}}\varphi \tilde{u}_{b,m}^{\beta-1}\tilde{u}_{b} \psi^2 \left(\int_{0}^{t_{1}}[-\dot{k}(\cdot+t_1-\tau)]u(\tau,x)d\tau\right)dx,
\;\;\;\;\mbox{a.a.}\;s\in(0,t_0-t_1).
\label{supnew6}
\end{align}

Assumption (H2) leads to
\[
\big(\tilde{A}D\tilde{u}_{b}|\psi^2
\tilde{u}_{b,m}^{\beta-1}[(\beta-1)D\tilde{u}_{b,m} + D\tilde{u}_{b}]\big) \geq (\beta-1)\nu \abs{D\tilde{u}_{b,m}}^2 u_{b,m}^{\beta-1}\psi^2 + \nu \abs{D\tilde{u}_{b}}^2 u_{b,m}^{\beta-1}\psi^2,
\]

while (H1) and Young's inequality give
\begin{align}
2\big|\big(\tilde{A}D\tilde{u}_{b}|\psi D\psi
\,\tilde{u}_{b,m}^{\beta-1}\tilde{u}_{b}\big)\big| & \le 2\Lambda\psi|D\psi|\,|D
\tilde{u}_{b}|\tilde{u}_{b,m}^{\beta-1} \tilde{u}_{b} \nonumber\\
& \le \,\frac{\nu}{2}\, \psi^2 |D \tilde{u}_{b}|^2
\tilde{u}_{b,m}^{\beta-1}+\,\frac{1}{\nu }\,\Lambda^2
|D\psi|^2 \tilde{u}_{b,m}^{\beta-1}\tilde{u}_b^2 . \label{supnew8}
\end{align}
Combining the previous estimates, we obtain
\begin{align}
& \int_{B_{r\sigma}}\varphi\psi^2H_m(\tilde{u}_{b})\,dx+
\nu \,l\ast\int_{B_{r\sigma}}\big( (\beta-1) \abs{D\tilde{u}_{b,m}}^2 u_{b,m}^{\beta-1}\psi^2 + \frac{1}{2} \abs{D\tilde{u}_{b}}^2 u_{b,m}^{\beta-1}\psi^2\big)\varphi\,dx\nonumber\\
\le \, &
\,\frac{2\Lambda^2}{\nu}\,l\ast\int_{B_{r\sigma}}\big(|D\psi|^2 \tilde{u}_{b,m}^{\beta-1}\tilde{u}_b^2\big)\varphi\,dx+\,l\ast \big(\int_{B_{r\sigma}}\psi^2\tilde{u}_{b,m}^{\beta-1}\tilde{u}_{b}u_{0}\, dx\,
k(\cdot+t_1)\varphi \big)\nonumber\\
& +\int_0^s l(s-\tau)\dot{\varphi}(\tau)
\big(k_{n}\ast
\int_{B_{r\sigma}}\psi^2H_m(\tilde{u}_{b})\,dx\big)(\tau)\,d\tau +b \int_{0}^{s}l(s-p)k(p)\vf(p)\int_{B_{r \sigma}} \psi^{2}\tilde{u}_{b,m}^{\beta-1}\tilde{u}_{b}dx dp \nonumber\\
& + l*\int_{B_{r\sigma}}\varphi \tilde{u}_{b,m}^{\beta-1}\tilde{u}_{b} \psi^2\int_{0}^{t_{1}}[-\dot{k}(\cdot+t_1-\tau)]u(\tau,x)\,d\tau dx,
\;\;\;\;\mbox{a.a.}\;s\in(0,t_0-t_1).
\label{new10}
\end{align}
Putting $w=\tilde{u}_{b,m}^{\frac{\beta-1}{2}}\tilde{u}_b$, we have
\[
Dw=\frac{\beta-1}{2}\tilde{u}_{b,m}^{\frac{\beta-1}{2}} D\tilde{u}_{b,m} + \tilde{u}_{b,m}^{\frac{\beta-1}{2}}D\tilde{u}_b,
\]
\[ |Dw|^2 \leq \frac{(\beta-1)^2}{2}\tilde{u}_{b,m}^{\beta-1} |D\tilde{u}_{b,m}|^2 + 2\tilde{u}_{b,m}^{\beta-1}|D\tilde{u}_b|^2 \leq 4\beta\left[ (\beta-1) \abs{D\tilde{u}_{b,m}}^2 u_{b,m}^{\beta-1} + \frac{1}{2} \abs{D\tilde{u}_{b}}^2 u_{b,m}^{\beta-1}\right].
\]
Furthermore, a direct calculation shows that
\[
w^2 \leq (1+\beta)H_m(\tilde{u}_b) \leq (1+\beta) w^2.
\]
Applying these estimates in (\ref{new10}) we obtain
\begin{equation} \label{sup9a}
\int_{B_{r\sigma}}\varphi\psi^2w^2\,dx+\,\frac{\nu
(\beta+1)}{4\beta}\,l\ast\int_{B_{r\sigma}}\varphi \psi^2|Dw|^2\,dx
\le l\ast F,\quad\mbox{a.a.}\;s\in(0,t_0-t_1),
\end{equation}
where
\eqnsl{
 F(s) & =\,  \,\frac{2\Lambda^2(1+\beta)}{\nu }\,  \int_{B_{r\sigma}}
|D\psi|^2\varphi w^2\,dx\\
&\quad+(1+\beta) \varphi(s)k(s)\int_{B_{r \sigma}} \psi^{2}w^{2} dx
+(1+\beta)\dot{\varphi}(s)
\big(k\ast
\int_{B_{r\sigma}}\psi^2 w^{2}\,dx\big)(s) \\
&
\quad+ \frac{(1+\beta)}{b}\left[\,\int_{B_{r\sigma}}\psi^2 w^{2} u_{0}\, dx\,
k(s+t_1)\varphi + \int_{B_{r\sigma}}\varphi w^{2} \psi^2\int_{0}^{t_{1}}[-\dot{k}(s+t_1-\tau)]u(\tau,x)d\tau dx\right].
}{defeFa}

Then by Young's inequality for convolutions and the properties of
$\varphi$ we infer that for all $1< p\leq p_0$, where $p_0$ comes from the assumption (\ref{ak1}),
\begin{equation} \label{sup10}
\Big(\int_{t_2-t_1}^{t_0-t_1} (\int_{B_{r\sigma}}
[\psi(x)w(s,x)]^2\,dx)^p\,ds\Big)^{1/p} \,\le
\|l\|_{L_p([0,t_0-t_1])} \int_0^{t_0-t_1} \!\!\!\!F(s)\,ds.
\end{equation}
We choose any of these $p$ and fix it.
Furthermore, the estimate (\ref{sup9a}) gives
\begin{equation} \label{sup12}
\int_{t_2-t_1}^{t_0-t_1}\int_{B_{r\sigma}}\psi^2|Dw|^2\,dx\,ds \le
\,\frac{ 4\beta}{\nu (\beta+1)}\,\int_0^{t_0-t_1} \!\!\!\!F(s)\,ds.
\end{equation}
\nic{
Let us estimate $\int_0^{t_0-t_1}F(s)ds$.
Next, we note that, by Lemma \ref{kernels} and Lemma \ref{ba} we have
\[
\vf(s)k(s) \leq k\left(\frac{t_{2}-t_{1}}{2}\right) \leq k_1\left(\jd (\rho-\rho') \sigma \eta \vdr  \right) \leq \frac{2 \max\{1,\eta^{-1}\}}{\sigma (\rho-\rho') } k_{1}(\vdr).
\]
Thus, we obtain
\eqq{
\vf(s)k(s)\int_{B_{r\sigma}}\psi^{2}w^{2}dx \leq c(\eta,\delta)(\rho-\rho')^{-1}k_1(\vdr) \int_{B_{\rho r\sigma}}w^{2}dx.
}{fa}
Further,
\[
\dot{\vf}(s)(k*\int_{B_{r\sigma}}\psi^{2}w^{2}dx)(s)\leq \frac{4}{\sigma\eta(\rho-\rho')\vdr}(k*\int_{B_{\rho r\sigma }}w^{2}dx)(s)
\]
and consequently, by Young inequality for convolution we get
\[
\int_{0}^{t_{0}-t_{1}}\dot{\vf}\cdot k*\int_{B_{\rho r\sigma }}w^{2}dxd\tau \leq  \frac{4}{\sigma\eta(\rho-\rho')\vdr} (1*k*\int_{B_{\rho r\sigma }}w^{2}dx)(t_0-t_1)
\]
\[
 \leq \frac{4}{\sigma\eta(\rho-\rho')\vdr} (1*k)(\rho\sigma\eta\vdr)\int_{0}^{t_{0}-t_{1}}\int_{B_{\rho r\sigma }}w^{2}dxd\tau
\]
\[
\leq \frac{4}{(\rho-\rho')}  \frac{(1*k)(\sigma\eta \vdr)}{\sigma\eta \vdr}\int_{0}^{t_{0}-t_{1}}\int_{B_{\rho r\sigma }}w^{2}dxd\tau
\leq  \frac{4 \overline{c} \max\{1,\eta^{-1}\}}{\sigma(\rho-\rho')} \ki(\vdr)
 \int_{0}^{t_{0}-t_{1}}\int_{B_{\rho r\sigma }}w^{2}dxd\tau,
\]
where in the last estimate we used the fact that $k$ is nonincreasing and Lemma \ref{kernels}.
Then,
\eqq{
\int_{0}^{t_{0}-t_{1}}\int_{B_{r\sigma}}\abs{D \psi}^{2}w^{2}dxds \leq \frac{4}{r^{2}\sigma^{2}(\rho-\rho')^{2}}\int_{0}^{t_{0}-t_{1}}\int_{B_{\rho r\sigma}}w^{2}dxds.
}{pierw1}}
Similarly as in (\ref{uz14}) and (\ref{uz15}) we have
\[
\frac{1}{b}\int_{0}^{t_0-t_1} \int_{B_{r\sigma}}\psi^2 w^{2}(s) u_{0}\, dx\, k(s+t_1)\varphi(s)ds \leq \frac{k(t_1)}{b}\int_{B_{r\sigma}}u_{0}dx \int_{0}^{t_0-t_1}  \norm{\psi w}^{2}_{L_{\infty}(B_{r\sigma})}ds
\]
\[
\leq
r^{-1}\int_{0}^{t_0-t_1} \norm{ \psi w}^{2}_{L_{\infty}(B_{r\sigma})}ds
\]
\nic{We come to the last term
We note that for $s\in (\frac{t_2-t_1}{2},t_0-t_1)$ we have
\[
-\dot{k}(s+t_1-\tau) =\frac{\al}{\Gamma(1-\al)}(s+t_1-\tau)^{-\al-1} \leq \frac{\al}{\Gamma(1-\al)} (t_0-\tau - \sigma\eta\vdr\frac{\rho+\rho'}{2})^{-\al-1}
\]
\[
=\frac{\al}{\Gamma(1-\al)}(t_0-\tau)^{-\al-1}(1-\frac{\sigma\eta\vdr(\rho+\rho')}{2(t_0-\tau)})^{-\al-1} \leq  -\dot{k}(t_0-\tau)(1-\frac{\rho+\rho'}{2\rho})^{-\al-1} \leq \frac{4}{(\rho'-\rho)^{2}}[-\dot{k}(t_0-\tau)]
\]
Hence, using $\rho \geq \frac{\sigma'}{\sigma}$ we have} and
\[
\frac{1}{b} \int_{0}^{t_0-t_1}\int_{B_r\sigma}\varphi(s) w^{2}(s) \psi^2\int_{0}^{t_{1}}[-\dot{k}(s+t_1-\tau)]u(\tau,x)\,d\tau dxds
\]
\[
\leq \frac{1}{b}\frac{c(\eta)}{(\rho'-\rho)^{\omega}} \int_{0}^{t_0-t_1}\int_{B_r\sigma}\varphi(s) w^{2}(s) \psi^2 \int_{0}^{t_0-\rho\sigma\eta\vdr}[-\dot{k}(t_0-\tau)]u(\tau,x)\,d\tau dxds
\]
\[
\leq \frac{1}{b}\frac{c(\eta)}{(\rho'-\rho)^{\omega}} \int_{0}^{t_0-t_1}\norm{\psi w }^2_{L_{\infty}(B_{r\sigma })} ds\int_{0}^{t_0-\sigma'\eta\vdr}[-\dot{k}(t_0-\tau)]\int_{B_{r\sigma}}u(\tau,x)\,dxd\tau.
\]
Thus, from the definition of $b$ we get
\[
\frac{1}{b} \int_{0}^{t_0-t_1}\int_{B_r\sigma}\varphi w^{2} \psi^2\int_{0}^{t_{1}}[-\dot{k}(s+t_1-\tau)]u(\tau,x)d\tau dx ds  \leq \frac{1}{r}\frac{c(\eta)}{(\rho'-\rho)^{\omega}}\int_{0}^{t_0-t_1}\norm{\psi w}_{L_{\infty}(B_{r\sigma})}^{2}ds.
\]

Estimating the remaining terms in (\ref{defeFa}) as in the proof of Theorem \ref{subest2}, we arrive at
\begin{align*}
\int_{0}^{t_0-t_1}\!\!\!\!\! F(s)\,ds  \leq c(\eta,\delta,\Lambda,\nu,\overline{c}) \frac{2+\beta}{(\rho-\rho')^{\tilde{\omega}}} \left(\frac{1}{r^{2}}\int_{0}^{t_0-t_1}\!\!\int_{B_{ r\sigma }}\!\!\!\psi^2 w^{2}dxd\tau+ r^{-1}\norm{\psi w}_{L_{2}((0,t_0-t_1);L_{\infty}(B_{ r\sigma }))}^{2} \right),
\end{align*}
where we denoted $\tilde{\omega} = \max\{2,\omega\}$.
Hence
\eqq{
\int_{0}^{t_0-t_1}F(s)\,ds  \leq c(\eta,\delta,\Lambda,\nu,\overline{c}) \frac{2+\beta}{(\rho-\rho')^{\tilde{\omega}}} r^{-1}\norm{\psi w}_{L_{2}((0,t_0-t_1);L_{\infty}(B_{ r\sigma }))}^{2}.
}{estiF1}
Combining (\ref{sup10}) and   (\ref{estiF1}) we obtain
\[
\norm{\psi w}_{L_{2p}((t_{2}-t_{1},t_{0}-t_{1});L_{2}(B_{r\sigma}))} \leq \norm{l}_{L_{p}((0,t_{0}-t_{1}))}^{\frac{1}{2}}\left(\int_{0}^{t_{0}-t_{1}}F(s)ds\right)^{\frac{1}{2}}
\]
\[
\leq C(\eta,\delta,\Lambda, \nu,\overline{c})\|l\|_{L_{p}((0,\rho\sigma\eta\vdr))}^{\frac{1}{2}}\frac{2+\beta}{(\rho-\rho')^{\tilde{\omega}/2}} r^{-1/2}\norm{\psi w}_{L_{2}((0,t_0-t_1);L_{\infty}(B_{r\sigma }))}.
\]
Since $l$ is monotone, we may estimate
\eqq{
\norm{l}^{p}_{L_{p}((0,\rho\sigma\eta\vdr))} \leq \max\{1,\eta\} \norm{l}^p_{L_{p}((0,\vdr))} \leq c(\eta)r^{2p}(\vdr)^{1-p},
}{estiPhieta}
where in the second estimate we applied Lemma \ref{scaling}.
Therefore, we have
\eqnsl{ & \norm{\psi w}_{L_{2p}((t_{2}-t_{1},t_{0}-t_{1});L_{2}(B_{r\sigma}))}   \\
& \leq C(\eta,\delta,\Lambda, \nu,p,\overline{c})  \frac{2+\beta}{(\rho-\rho')^{\tilde{\omega}/2}}r^{1/2}(\vdr)^{-\frac{1}{2p'}}\norm{\psi w}_{L_{2}((0,t_0-t_1);L_{\infty}(B_{r \sigma }))}.}{estiwwl2p}
Furthermore, from (\ref{sup12}),  (\ref{estiF1})  and
\[
\int_{t_2-t_1}^{t_0-t_1}\int_{B_{r\sigma}} |D(\psi w)|^2\,dx\,ds\le
2\int_{t_2-t_1}^{t_0-t_1}\int_{B_{r\sigma}}
\big(\psi^2|Dw|^2+|D\psi|^2w^2\big)\,dx\,ds
\]
we infer that
\eqnsl{ &
\norm{D(\psi w)}_{L_{2}((t_{2}-t_{1},t_{0}-t_{1});L_{2}(B_{r\sigma}))} \\
&\le
C(p,\eta,\Lambda, \delta,\nu,\overline{c})\frac{2+\beta}{(\rho-\rho')^{\tilde{\omega}/2}}r^{-1/2}\norm{w}_{L_{2}((0,t_0-t_1);L_{\infty}(B_{r \rho\sigma }))}.
}{estiDpsiw}

Let us denote
\eqq{
\kappa = 1+ \frac{1}{2p'-1}.
}{defkappatheta}
Applying the embedding (\ref{sobolev}) with $a = 2\kappa$ and $b = \infty$ we obtain
\eqnsl{
& \norm{\psi w}_{L_{2\kappa}((t_{2}-t_{1},t_{0}-t_{1});L_{\infty}(B_{r\sigma}))}   \\
&\quad \quad\leq C \norm{D(\psi w)}^{\frac{1}{2}}_{L_{2}((t_{2}-t_{1},t_{0}-t_{1});L_{2}(B_{r\sigma}))}  \norm{\psi w}_{L_{2p}((t_{2}-t_{1},t_{0}-t_{1});L_{2}(B_{r\sigma}))}^{\frac{1}{2}}.
}{interapara}
Using in the above inequality the estimates (\ref{estiwwl2p})  and (\ref{estiDpsiw}), we arrive at
\eqq{
\norm{\psi w}_{L_{2\kappa}((t_{2}-t_{1},t_{0}-t_{1});
L_{\infty}(B_{r\sigma}))}
 \leq C(p,\eta,\Lambda, \delta,\nu,\overline{c})(\vdr)^{-\frac{1}{4p'}}\frac{2+\beta}{(\rho-\rho')^{\tilde{\omega}/2}}\norm{w}_{L_{2}((0,t_0-t_1);L_{\infty}(B_{\rho r\sigma }))}.
}{mainMoser1}

Now recall that $w = \tilde{u}_{b,m}^{\frac{\beta-1}{2}} \tilde{u}_b$. We may pass to the limit with $m$ in (\ref{mainMoser1}), using the monotone convergence theorem, the fact that $u \in L_{2}((0,T);L_{\infty}(\Omega))$ and Lemma \ref{moserit1}.
Then, denoting  $\gamma=1+\beta$ and employing  properties of $\psi$ we get
\[
\| u_b\|_{L_{ \gamma \kappa , \infty}(V_{\rho'})} \leq \left(\frac{C^{2}(1+\gamma)^{2}}{(\rho-\rho')^{\tilde{\omega}}(\vdr)^{\frac{1}{2p'}}}\right)^{\frac{1}{\gamma}}\| u_b\|_{L_{ \gamma ,\infty}(V_{\rho})}.
\]
We note that $\kappa>1$ and we apply the  Moser lemma (Lemma~\ref{moserit1}) with $\beta_1 = 1$,  $\beta_2=\infty$ to get
\eqq{
\esssup_{V_{\vsi}}u_b\leq \left(\frac{M_{0}}{(1-\vsi)^{\frac{\tilde{\omega}\kappa}{\kappa-1}}(\vdr)^{\frac{\kappa}{2p'(\kappa-1)}}}\right)^{\frac{1}{\gamma}}\| u_b\|_{L_{ \gamma  ,\infty}(V_{1})}, \hd \gamma\in (0,1], \hd \vsi \in \big[\frac{\sigma'}{\sigma},1\big),
}{zMoseranaub}
where $M_{0}=M_{0}(p ,\eta, \Lambda, \delta, \nu ,\overline{c})$.
Since $\kappa/(\kappa-1) = 2p'$
 we obtain
\[
\esssup_{V_{\vsi}}u_b\leq \left(\frac{M_{0}}{(1-\vsi)^{\frac{\tilde{\omega}\kappa}{\kappa-1}}\vdr}\right)^{\frac{1}{\gamma}}  \| u_b \|_{L_{ \gamma  ,\infty}(V_{1})}, \hd \gamma\in (0,1], \hd
\vsi \in \big[\frac{\sigma'}{\sigma},1\big).
\]
Thus, if we take $\gamma=1$,
$\vsi=\sigma'/\sigma$ and notice that
\[ \frac{1}{1-\vsi}\,=\,\frac{\sigma}{\sigma-\sigma'}\,\le
\frac{1}{\sigma-\sigma'},\]
then we obtain for $r \leq r^{*}$
\[
\esup_{U_{\sigma'}}{u_b} \le
\frac{M_0 }{(\sigma-\sigma')^{\tau_0}}\frac{1}{\vdr}
\int_{t_0-\sigma\eta\vdr}^{t_0}\sup_{x\in B_{r\sigma }}u_b(t,x)dt,
\]
which finishes the proof.
\begin{satz}\label{subestimate}
Let $T>0$ and $\Omega\subset \iR$ be a bounded
domain. Suppose the assumptions (H1)--(H3) and (\ref{pc}), (\ref{ak1}), (\ref{asholder}) are satisfied. Let
$\eta>0$ and $\delta\in (0,1)$ be fixed. Then there exists $r^{*}>0$ such that for any $0<r \leq r^{*}$, any $t_0\in(0,T]$
 with $t_0-\eta \vdr \ge 0$, any ball
$B_r(x_{0})\subset\Omega$, and any weak subsolution $u\ge 0$ of (\ref{MProb}) in $(0,t_0)\times B_r(x_{0})$ with $u_0\ge 0$
in $B_r(x_{0})$, there holds
\[
\esup_{U_{\sigma'}}{u} \le C(\sigma-\sigma')^{-2\tau_1}\frac{1}{r\vdr}  \int_{t_0-\sigma\eta\vdr}^{t_0} \int_{B_{r\sigma }} u\, dxdt
\]
\[
+\frac{C}{(\sigma-\sigma')^{\tau_1}} \left(k(t_0-\sigma\eta\vdr)r\int_{B{r\sigma }} u_0\, dx+r\int_{0}^{t_0-\sigma\eta\vdr}[-\dot{k}(t_0-\tau)]\int_{B_{r\sigma}}u(\tau,x)\,dxd\tau \right),
\]
where $U_\sigma=(t_0-\sigma\eta \vdr,t_0)\times B_{r\sigma }(x_0)$,
$0<\delta\leq \sigma'<\sigma\le 1$, $C=C(\nu,\Lambda,\delta,\eta,\overline{c},c_0)$ and
$\tau_1=\tau_{1}(\omega)>0$
\end{satz}

\begin{proof}
Fix $\delta<\sigma'<\sigma <  1$ and apply Theorem  \ref{subest1} with parameters $\sigma',\frac{\sigma+\sigma'}{2}$ and then Theorem \ref{subest2} with parameters $\frac{\sigma'+\sigma}{2},\sigma$. Setting $\tau_1 = \tau_0+1$ we thus obtain
\[
\esup_{U_{\sigma'}}{u} \leq  \frac{C}{(\sigma-\sigma')^{\tau_1}} \left(\frac{1}{r\vdr}\int_{t_0-\sigma\eta\vdr}^{t_0} \int_{B_{r\sigma }} u^{2}dxdt\right)^{\frac{1}{2}}
\]
\eqq{
+\frac{C}{(\sigma-\sigma')^{\tau_1}} \left( k(t_0-\sigma\eta\vdr)r \int_{B_{r\sigma}} u_0 dx +r\int_{0}^{t_0-\sigma'\eta\vdr}[-\dot{k}(t_0-\tau)]\int_{B_{r\sigma}}u(\tau,x)dxd\tau\right).
}{supltwo}

By Young's inequality applied to (\ref{supltwo}) we further have
\[
\esup_{U_{\sigma'}}{u} \le \frac{1}{2}\esup_{U_{\sigma}}{u} +C(\sigma-\sigma')^{-2\tau_1}  \frac{1}{r\vdr}\int_{t_0-\sigma\eta\vdr}^{t_0} \int_{B_{r\sigma }} u\,dxdt
\]
\[
 +\frac{C}{(\sigma-\sigma')^{\tau_1}} \left( k(t_0-\sigma\eta\vdr)r \int_{B_{r\sigma}} u_0\, dx +r\int_{0}^{t_0-\sigma'\eta\vdr}[-\dot{k}(t_0-\tau)]\int_{B_{r\sigma}}u(\tau,x)\,dxd\tau\right).
\]
Let us take $\sigma
\leq \tilde{\sigma} \leq 1$. Then, by convexity of $k$ we may estimate
\[
\int_{0}^{t_0-\sigma'\eta\vdr}[-\dot{k}(t_0-\tau)]\int_{B_{r\sigma}}u(\tau,x)dxd\tau
\]
\[
\leq \int_{0}^{t_0-\tilde{\sigma}\eta\vdr}[-\dot{k}(t_0-\tau)]\int_{B_{r\tilde{\sigma}}}u(\tau,x)dxd\tau +
\int_{t_0-\tilde{\sigma}\eta\vdr}^{t_0-\sigma'\eta\vdr}[-\dot{k}(t_0-\tau)]\int_{B_{r\sigma}}u(\tau,x)dxd\tau
\]
\[
\leq \int_{0}^{t_0-\tilde{\sigma}\eta\vdr}[-\dot{k}(t_0-\tau)]\int_{B_{r\tilde{\sigma}}}u(\tau,x)dxd\tau + [-\dot{k}(\sigma'\eta\vdr)]
\int_{t_0-\tilde{\sigma}\eta\vdr}^{t_0}\int_{B_{r\tilde{\sigma}}}u(\tau,x)dxd\tau.
\]
Applying the assumption (\ref{asholder}) together with Lemma \ref{fi} and Lemma \ref{ba} we find that
\[
\int_{0}^{t_0-\sigma'\eta\vdr}[-\dot{k}(t_0-\tau)]\int_{B_{r\sigma}}u(\tau,x)dxd\tau
\]
\[
\leq \int_{0}^{t_0-\tilde{\sigma}\eta\vdr}[-\dot{k}(t_0-\tau)]\int_{B_{r\tilde{\sigma}}}u(\tau,x)dxd\tau +\frac{C(\delta,\eta,c_0)}{\vdr r^{2}} \int_{t_0-\tilde{\sigma}\eta\vdr}^{t_0}\int_{B_{r\tilde{\sigma}}}u(\tau,x)dxd\tau.
\]
Hence, we obtain that for every $0<\delta\leq \sigma'<\sigma \leq \tilde{\sigma} \leq 1$
\[
\esup_{U_{\sigma'}}{u} \le \frac{1}{2}\esup_{U_{\sigma}}{u} +\frac{C(\sigma-\sigma')^{-2\tau_1}}{\vdr r}  \int_{t_0-\tilde{\sigma}\eta\vdr}^{t_0} \int_{B_{\tilde{\sigma} r}} u\,dxdt
\]
\[
+\frac{C}{(\sigma-\sigma')^{\tau_1}} \left(k(t_0-\tilde{\sigma}\eta\vdr)r \int_{B_{\tilde{\sigma} r}} u_0 \,dx+r\int_{0}^{t_0-\tilde{\sigma}\eta\vdr}[-\dot{k}(t_0-\tau)]\int_{B_{r\tilde{\sigma}}}u(\tau,x)\,dxd\tau \right).
\]
From (\ref{supltwo}) we infer that $u$ is locally bounded in $\Omega_{T}$, hence applying Lemma \ref{geometric} with $f(\sigma) = \esup_{U_{\sigma}}{u}$  we obtain the claim.

\end{proof}
\begin{bemerk1}
{\em The assumption that the space dimension is one
is again crucial for using the interpolation inequality, cf.\ Remark \ref{RemA}. For this reason the subsolution estimates cannot be extended to higher dimensions.
}
\end{bemerk1}

Theorem \ref{strongHarnack} follows from Theorem \ref{imweakHarnack} applied together with Theorem \ref{subestimate} with appropriately modified cylinders.

${}$


\end{document}